\documentclass[12pt]{article}%
\usepackage[intlimits]{amsmath}
\usepackage{amssymb}
\usepackage[T1]{fontenc}
\usepackage[sc]{mathpazo}
\usepackage{color}
\usepackage[colorlinks]{hyperref}
\usepackage{amsfonts}
\usepackage{graphicx}%
\definecolor {refcol}{RGB}{40,0,255}
\hypersetup{colorlinks=true,allcolors=refcol}
\makeatletter
\newfont{\footsc}{cmcsc10 at 8truept}
\newfont{\footbf}{cmbx10 at 8truept}
\newfont{\footrm}{cmr10 at 10truept}
\newtheorem{theorem}{Theorem}[section]

\newtheorem{conjecture}[theorem]{Conjecture}
\newtheorem{corollary}[theorem]{Corollary}

\newtheorem{definition}[theorem]{Definition}

\newtheorem{problem}[theorem]{Problem}
\newtheorem{proposition}[theorem]{Proposition}
\newtheorem{question}[theorem]{Question}

\newenvironment{proof}[1][Proof]{\noindent{\textbf {#1}  }}  {\hfill$\Box$\bigskip}

\begin{document}

\title{\textbf{Beyond graph energy: norms of graphs and matrices}}
\author{V. Nikiforov\thanks{Department of Mathematical Sciences, University of
Memphis, Memphis TN 38152, USA}}
\date{}
\maketitle

\begin{abstract}
In 1978 Gutman introduced the energy of a graph as the sum of the absolute
values of graph eigenvalues, and ever since then graph energy has been
intensively studied.

Since graph energy is the trace norm of the adjacency matrix, matrix norms
provide a natural background for its study. Thus, this paper surveys research
on matrix norms that aims to expand and advance the study of graph energy.

The focus is exclusively on the Ky Fan and the Schatten norms, both
generalizing and enriching the trace norm. As it turns out, the study of
extremal properties of these norms leads to numerous analytic problems with
deep roots in combinatorics.

The survey brings to the fore the exceptional role of Hadamard matrices,
conference matrices, and conference graphs in matrix norms. In addition, a
vast new matrix class is studied, a relaxation of symmetric Hadamard matrices.

The survey presents solutions to just a fraction of a larger body of similar
problems bonding analysis to combinatorics. Thus, open problems and questions
are raised to outline topics for further investigation.\textit{\medskip}

\textbf{Keywords: }\textit{graph energy; graph norm; Ky Fan norm; Schatten
norm; Hadamard matrix; conference matrix.\medskip}


\newpage

\end{abstract}
\tableofcontents

\newpage

\section{Introduction}

This paper overviews the current research on the Ky Fan and the Schatten
matrix norms that aims to expand and push forward the study of graph energy.

Graph energy has been introduced by Gutman in 1978 \cite{Gut78} \ as the sum
of the absolute values of the graph eigenvalues; since then its study has
produced a monumental body of work, as witnessed by the references of the
monograph \cite{GLS12}. One reckons that such an enduring interest must be
caused by some truly special property of the graph energy parameter.

\subsection{Graph energy as a matrix norm}

To crack the mystery of graph energy, it may be helpful to view it as the
trace norm of the adjacency matrix. Recall that the \emph{trace norm}
$\left\Vert A\right\Vert _{\ast}$ of a matrix $A$ is the sum of its singular
values, which for a real symmetric matrix are just the moduli of its eigenvalues.

Hence, if $G$ is graph with adjacency matrix $A,$ then the energy of $G$ is
the\ trace norm of $A.$

This simple observation, made in \cite{Nik07i}, triggered some sort of a chain
reaction. On the one hand, usual tools for norms provided new techniques for
graph energy; see, e.g., \cite{AGO09}, \cite{DaSo07}, \cite{DaSo08}, and
\cite{SRAG10}. On the other hand, viewed \ as a trace norm, graph energy was
extended to non-symmetric and even to non-square matrices, and gave rise to
new topics like \textquotedblleft skew energy\textquotedblright\ and
\textquotedblleft incidence energy\textquotedblright, see, e.g., \cite{ABW10},
\cite{GKM09}, \cite{GKMZ09}, \cite{JKM09}, \cite{ZhLi12}, \cite{ZKL15}, and
\cite{Zho10}.\medskip

To follow this track further, we need a few definitions. Write $M_{m,n}$ for
the space of the $m\times n$ complex matrices.

\begin{definition}
A \textbf{matrix norm} is a nonnegative function $\left\Vert \cdot\right\Vert
$ defined on $M_{m,n}$ such that:

(a) $\left\Vert A\right\Vert =0$ if and only if $A=0$;

(b) $\left\Vert cA\right\Vert =\left\vert c\right\vert \left\Vert A\right\Vert
$ for every complex number $c$;

(c) $\left\Vert A+B\right\Vert \leq\left\Vert A\right\Vert +\left\Vert
B\right\Vert $ for every $A,B\in M_{m,n}.$
\end{definition}

Observe that the usual definition of matrix norm includes additional
properties, but they are not used in this survey.\medskip

A simple example of a matrix norm is the \emph{max-norm} $\left\Vert
A\right\Vert _{\max}$ defined for any matrix $A:=\left[  a_{i,j}\right]  $ as
$\left\Vert A\right\Vert _{\max}:=\max_{i,j}\left\vert a_{i,j}\right\vert .$

Another example of a matrix norm is the largest singular value of a matrix,
also known as its \emph{operator norm}.

By contrast, the spectral radius $\rho\left(  A\right)  $ of a square matrix
$A$ is not a matrix norm. Likewise, if $\lambda_{1}\left(  A\right)
,\ldots,\lambda_{n}\left(  A\right)  $ are the eigenvalues of $A,$ neither of
the functions
\[
g\left(  A\right)  :=\left\vert \lambda_{1}\left(  A\right)  \right\vert
+\cdots+\left\vert \lambda_{n}\left(  A\right)  \right\vert \text{ \ \ and
\ \ }h\left(  A\right)  :=\left\vert \operatorname{Re}\lambda_{1}\left(
A\right)  \right\vert +\cdots+\left\vert \operatorname{Re}\lambda_{n}\left(
A\right)  \right\vert
\]
is a matrix norm. Indeed, if $A$ is an upper triangular\ matrix with zero
diagonal, then $\rho\left(  A\right)  =g\left(  A\right)  =$\ $h\left(
A\right)  =0,$ violating condition (a). Worse yet, if $A$ is the adjacency
matrix of a nonempty graph $G,$ and if $A_{U}$ and $A_{L}$ are the upper and
lower triangular parts of $A,$ then $A=A_{U}+A_{L}$, whereas $\rho\left(
A\right)  >\rho(A_{U})+\rho(A_{L})=0,$ $g\left(  A\right)  >g(A_{U}%
)+g(A_{L})=0,$ and $h\left(  A\right)  >h(A_{U})+h(A_{L})=0,$ violating the
triangle inequality.

Notwithstanding the importance of $\rho\left(  A\right)  ,$ $g\left(
A\right)  $ and $h\left(  A\right)  $, we reckon that if a matrix parameter
naturally extends graph energy, it must obey at least conditions
(a)--(c).\medskip

Further, any matrix norm $\left\Vert \cdot\right\Vert $ defined on $M_{n,n}$
can be extended to graphs of order $n$ by setting $\left\Vert G\right\Vert
:=\left\Vert A\right\Vert ,$ where $A$ is the adjacency matrix of $G$. Thus,
we write $\left\Vert G\right\Vert _{\ast}$ for the energy of a graph
$G.$\medskip

In this survey we deal exclusively with the Ky Fan and the Schatten norms,
both defined via singular values. Recall that the \emph{singular values} of a
matrix $A$ are the square roots of the eigenvalues of $A^{\ast}A,$ where
$A^{\ast}$ is the conjugate transpose of $A.$ We write $\sigma_{1}\left(
A\right)  ,\sigma_{2}\left(  A\right)  ,\ldots$ for the singular values of $A$
arranged in descending order.$\medskip$

We write $I_{n}$ for the identity matrix of order $n,$ and $J_{m,n}$ for the
$m\times n$ matrix of all ones; we let $J_{n}=J_{n,n}.$ The \emph{Kronecker
product} of matrices is denoted by $\otimes.$ Recall that the singular values
of $A\otimes B$ are the products of the singular values of $A$ and the
singular values of $B,$ with multiplicities counted.\medskip

Finally, we call a matrix \emph{regular}, if its row sums are equal, and so
are its column sums. Note that the adjacency matrix of a regular graph is
regular, and the biadjacency matrix of a bipartite semiregular graph is also
regular; these facts explains our choice for the term \textquotedblleft
regular\textquotedblright. It is not hard to show that a matrix $A\in M_{m,n}$
is regular if and only if $A$ has a singular value with singular vectors that
are collinear to the all-ones vectors $\mathbf{j}_{n}\in\mathbb{R}^{n}$ and
$\mathbf{j}_{m}\in\mathbb{R}^{m}$.\medskip

\emph{Weyl's inequality for singular values\medskip}

For reader's sake we state Weyl's inequality for the singular values of sums
of matrices (see, e.g., \cite{HoJo12}, p. 454), which we shall use on several
occasions:\medskip

\emph{Let }$n\geq m\geq1,$ $A\in M_{m,n}$\emph{ and }$B\in M_{m,n}.$\emph{ If
}$i\geq1,$\emph{ }$j\geq1,$\emph{ and }$i+j\leq m+1,$\emph{ then}%
\begin{equation}
\sigma_{i+j-1}\left(  A+B\right)  \leq\sigma_{i}\left(  A\right)  +\sigma
_{j}\left(  B\right)  . \label{Wsin}%
\end{equation}

\emph{The Power Mean inequality}\textbf{\medskip}

For convenience, we also state the \emph{Power Mean (=PM) inequality:}\medskip

\emph{If }$q>p>0,$\emph{ and }$x_{1},\ldots,x_{n}$\emph{ are nonnegative real
numbers}$,$\emph{ then}%
\[
\left(  \frac{x_{1}^{p}+\cdots+x_{n}^{p}}{n}\right)  ^{1/p}\leq\left(
\frac{x_{1}^{q}+\cdots+x_{n}^{q}}{n}\right)  ^{1/q},
\]
\emph{with equality holding if and only if }$x_{1}=\cdots=x_{n}.\medskip$

If $p=1$ and $q=2,$ the PM inequality is called the \emph{Arithmetic
Mean-Quadratic Mean (=AM-QM) inequality}.

\subsection{\label{KMn}The Koolen and Moulton bound via norms}

To illustrate the extreme fitness of matrix norms for the study of graph
energy, we shall use norms to derive and extend the well-known result of
Koolen and Moulton \cite{KoMo01}:\medskip

\emph{If }$G$\emph{ is a graph of order }$n,$\emph{ then }%
\begin{equation}
\left\Vert G\right\Vert _{\ast}\leq\frac{n\sqrt{n}}{2}+\frac{n}{2}, \label{KM}%
\end{equation}
\emph{with equality holding if and only if }$G$\emph{ belongs to a specific
family of strongly regular graphs.}\medskip

We shall show that in fact inequality (\ref{KM}) has nothing to do with
graphs. Indeed, suppose that $A$ is an $n\times n$ nonnegative matrix with
$\left\Vert A\right\Vert _{\max}\leq1,$ and let $H:=2A-J_{n}.$ Obviously,
$\left\Vert H\right\Vert _{\max}\leq1,$ and so the AM-QM inequality implies
that
\begin{align}
\left\Vert H\right\Vert _{\ast}  &  =\sigma_{1}\left(  H\right)
+\cdots+\sigma_{n}\left(  H\right)  \leq\sqrt{n\left(  \sigma_{1}^{2}\left(
H\right)  +\cdots+\sigma_{n}^{2}\left(  H\right)  \right)  }\label{in2}\\
&  =\sqrt{n\text{ }\mathrm{tr}\left(  HH^{\ast}\right)  }=\left\Vert
H\right\Vert _{2}\sqrt{n}\leq n\sqrt{n}. \label{in2.1}%
\end{align}
Now, the triangle inequality for the trace norm yields
\[
\left\Vert 2A\right\Vert _{\ast}=\left\Vert H+J_{n}\right\Vert _{\ast}%
\leq\left\Vert H\right\Vert _{\ast}+\left\Vert J_{n}\right\Vert _{\ast}\leq
n\sqrt{n}+n,
\]
and so,%
\begin{equation}
\left\Vert A\right\Vert _{\ast}\leq\frac{n\sqrt{n}}{2}+\frac{n}{2}. \label{N1}%
\end{equation}
Therefore, if $A$ is the adjacency matrix of a graph $G,$ then (\ref{KM})
follows. That's it.\medskip

Now, we give a necessary and sufficient condition for equality in (\ref{N1}):

\begin{proposition}
\label{proKM}If $A$ is an $n\times n$ nonnegative matrix with $\left\Vert
A\right\Vert _{\max}\leq1,$ then equality holds in (\ref{N1}) if and only if
the matrix $H:=2A-J_{n}$ is a regular Hadamard matrix.
\end{proposition}

\begin{proof}
If $A$ forces equality in (\ref{N1}), then equalities hold throughout
(\ref{in2}) and (\ref{in2.1}); hence, $H$ is a $\left(  -1,1\right)  $-matrix
and all singular values of $H$ are equal to $\sqrt{n}.$ Therefore, $H$ is an
Hadamard matrix (see Proposition \ref{pro1} for details).

To show that $H$ is regular, note that for every $i=2,\ldots,n,$ Weyl's
inequality (\ref{Wsin}) implies that%
\[
\sigma_{i}\left(  2A\right)  \leq\sigma_{i-1}\left(  H\right)  +\sigma
_{2}\left(  J_{n}\right)  =\sqrt{n}.
\]
Since equality holds in (\ref{N1}), we find that
\[
\sigma_{1}\left(  2A\right)  \geq n\sqrt{n}+n-\left(  n-1\right)  \sqrt
{n}=\sqrt{n}+n=\sigma_{1}\left(  H\right)  +\sigma_{1}\left(  J_{n}\right)  .
\]
On the other hand, letting $\mathbf{x}\in\mathbb{R}^{n}$ and $\mathbf{y}%
\in\mathbb{R}^{n}$ be unit singular vectors to $\sigma_{1}\left(  2A\right)
,$ we see that
\begin{equation}
\sigma_{1}\left(  2A\right)  =\left\langle 2A\mathbf{x},\mathbf{y}%
\right\rangle =\left\langle H\mathbf{x},\mathbf{y}\right\rangle +\left\langle
J_{n}\mathbf{x},\mathbf{y}\right\rangle \leq\sigma_{1}\left(  H\right)
+\sigma_{1}\left(  J_{n}\right)  . \label{in3}%
\end{equation}
The latter inequality holds, for if $B$ is an $m\times n$ real matrix, then
\[
\sigma_{1}\left(  B\right)  =\max\left\{  \left\langle B\mathbf{x}%
,\mathbf{y}\right\rangle :\mathbf{x}\in\mathbb{R}^{n},\text{ }\mathbf{y}%
\in\mathbb{R}^{m},\text{ }\left\Vert \mathbf{x}\right\Vert _{2}=\left\Vert
\mathbf{y}\right\Vert _{2}=1\right\}  .
\]
Therefore, equality holds in (\ref{in3}), and so $\mathbf{x}$ and $\mathbf{y}$
are singular vectors to $\sigma_{1}\left(  H\right)  $ and to $\sigma
_{1}\left(  J_{n}\right)  $ as well; hence, $\mathbf{x}=\mathbf{y}%
=n^{-1/2}\mathbf{j}_{n},$ and so $H$ is regular. This completes the proof of
the \textquotedblleft only if\textquotedblright\ part of Proposition
\ref{proKM}; we omit the easy \textquotedblleft if\textquotedblright\ part.
\end{proof}

For graphs, Proposition \ref{proKM} has to be modified accordingly:

\begin{corollary}
If $G$ is a graph of order $n$ with adjacency matrix $A,$ then $G$ forces
equality in (\ref{KM}) if and only if the matrix $H:=2A-J_{n}$ is a regular
symmetric Hadamard matrix, with\ $-1$ along the main diagonal.
\end{corollary}

Undeniably, the above arguments shed additional light on the inequality of
Koolen and Moulton (\ref{KM}). Indeed, we realize that it is not about
graphs---it is an analytic result about nonnegative matrices with bounded
entries, with no symmetry or zero diagonal required. We also see that the
triangle inequality is extremely efficient in graph energy problems, and can
save pages of calculations.

Finally and most importantly, Proposition \ref{proKM} (noted for graphs by
Haemers in \cite{Hae08}) gives a sharper description of the extremal graphs
and matrices, and exhibits the strong bonds between matrix norms and Hadamard
matrices. It should be noted that similar ideas have been outlined by
Bo\v{z}in and Mateljevi\'{c} in \cite{BoMa11}, but their study remained
restricted to symmetric matrices.

\subsection{The Ky Fan and the Schatten norms}

The trace norm is the intersection of two fundamental infinite classes of
matrix norms, namely the Schatten $p$-norms and the Ky Fan $k$-norms, defined
as follows:

\begin{definition}
Let $p\geq1$ be a real number and let $n\geq m\geq1.$ If $A\in M_{m,n},$ the
\textbf{Schatten }$p$\textbf{-norm} $\left\Vert A\right\Vert _{p}$ of $A$ is
given by%
\[
\left\Vert A\right\Vert _{p}:=\left(  \sigma_{1}^{p}\left(  A\right)
+\cdots+\sigma_{m}^{p}\left(  A\right)  \right)  ^{1/p}.
\]

\end{definition}

\begin{definition}
Let $n\geq m\geq k\geq1.$ If $A\in M_{m,n},$ the \textbf{Ky Fan }%
$k$\textbf{-norm} $\left\Vert A\right\Vert _{\left[  k\right]  }$ is given by%
\[
\left\Vert A\right\Vert _{\left[  k\right]  }:=\sigma_{1}\left(  A\right)
+\cdots+\sigma_{k}\left(  A\right)  .
\]

\end{definition}

As already mentioned, if $A\in M_{m,n}$, then
\[
\left\Vert A\right\Vert _{\ast}=\left\Vert A\right\Vert _{\left[  m\right]
}=\left\Vert A\right\Vert _{1}.
\]

Another important case is the Schatten $2$-norm $\left\Vert A\right\Vert
_{2},$ also known as the \emph{Frobenius norm} or $A.$ The norm $\left\Vert
A\right\Vert _{2}$ satisfies the following equality, which can be checked
directly%
\begin{equation}
\left\Vert A\right\Vert _{2}=\sqrt{\sigma_{1}^{2}\left(  A\right)
+\cdots+\sigma_{m}^{2}\left(  A\right)  }=\sqrt{\mathrm{tr}\text{ }\left(
AA^{\ast}\right)  }=\sqrt{\mathrm{tr}\text{ }\left(  A^{\ast}A\right)  }%
=\sqrt{\sum\nolimits_{i,j}\left\vert a_{ij}\right\vert ^{2}}. \label{meq}%
\end{equation}
Note that equality (\ref{meq}) is widely used in spectral graph theory, for if
a graph $G$ has $m$ edges, then $\left\Vert G\right\Vert _{2}=\sqrt
{2m}.\medskip$

Since both the Schatten and the Ky Fan norms generalize graph energy, in
\cite{Nik11c}, \cite{Nik12}, and \cite{Nik15c}, the author proposed to study
these norms for their own sake. Recently some progress was reported in
\cite{Nik15b} and \cite{NiYu13}.

In fact, certain Schatten norms of graphs have already been studied in graph
theory, albeit implicitly: if $A$ is the adjacency matrix of a graph $G,$ then
$\left\Vert G\right\Vert _{2}^{2}=\mathrm{tr}A\,^{2}=2e\left(  G\right)  ,$
and if $k\geq2,$ then $\left\Vert G\right\Vert _{2k}^{2k}=\mathrm{tr}%
\,A^{2k},$ and so $\left\Vert G\right\Vert _{2k}^{2k}/4k$ is the number of
closed walks of length $2k$ in $G.\medskip$

This survey shows that research on the Schatten and the Ky Fan norms adds
significant volume and depth to graph energy. Older topics are seen in new
light and from new viewpoints, but most importantly, this research leads to
deep and hard problems, whose solutions would propel the theory of graph
energy to a higher level.

\subsection{Extremal problems for norms}

Arguably the most attractive problems in spectral graph theory are the
extremal ones, with general form like:\medskip

\emph{If }$G$\emph{ is a graph of order }$n,$\emph{ with some property
}$\mathcal{P},$\emph{ how large or small can a certain spectral parameter }%
$S$\emph{ be}$?\medskip$

Such extremal questions are crucial to spectral graph theory, for they are a
sure way to connect the spectrum of a graph to its structure. Not
surprisingly, extremal problems are the hallmark of the study on graph energy
as well, with $S$ being the trace norm and $\mathcal{P}$ chosen from a huge
variety of graph families.

Thus, our survey keeps expanding and promoting this emphasis, with property
$\mathcal{P}$ chosen among the basic ones, like: \textquotedblleft\emph{all
graphs}\textquotedblright, \textquotedblleft\emph{bipartite}\textquotedblright%
, \textquotedblleft$r$\emph{-partite}\textquotedblright, \ and
\textquotedblleft$K_{r}$\emph{-free}\textquotedblright, whereas the spectral
parameter $S$ is a Schatten or a Ky Fan norm. Hence, upper and lower bounds on
these norms will be a central topic in our presentation.\medskip

Let us mention a striking tendency here: frequently global\ parameters like
norms are maximized on rare matrices of delicate structure, as it happens,
e.g., in (\ref{N1}). Naturally the possibility of finding such extremal
structures adds a lot of thrill and appeal to these problems, for they throw
bridges between distant corners of analysis and combinatorics.\medskip

We shall also discuss Nordhaus-Gaddum problems; to this end, let $\overline
{G}$ denote the complement of a graph $G.$ A \emph{Nordhaus-Gaddum problem} is
of the following type:\medskip

\emph{Given a graph parameter }$p\left(  G\right)  ,$\emph{ determine}
\[
\max\left\{  p\left(  G\right)  +p(\overline{G}):v\left(  G\right)
=n\right\}  .
\]

Nordhaus-Gaddum problems were introduced in \cite{NoGa56}, and subsequently
studied for numerous graph parameters; see \cite{AoHa13} for a comprehensive
survey. In \cite{NiYu13}, Nordhaus-Gaddum problems were studied for matrices,
with $p\left(  A\right)  $ being a Ky Fan norm. As it turns out, these
Nordhaus-Gaddum parameters are maximized on the adjacency matrices of
conference graphs, also of delicate and rare structure.

\subsection{Matrices vs. graphs}

Somewhat surprisingly, results about energy of graphs often extend
effortlessly to more general matrices, say, to nonnegative or even to complex
rectangular ones. Such extensions usually shed extra light, sometimes on the
original graph results, and sometimes on well-known matrix topics.

We saw such an extension in the case of inequality (\ref{KM}), and here is
another one. Recall that in \cite{Nik12} the McClelland upper bound
\cite{McL71} was extended to rectangular matrices as:

\begin{proposition}
\label{pro}If $n\geq m\geq2$ and $A\in M_{m,n},$ then
\begin{equation}
\left\Vert A\right\Vert _{\ast}\leq\sqrt{m\left\Vert A\right\Vert _{2}}.
\label{hin0}%
\end{equation}
If in addition $\left\Vert A\right\Vert _{\max}\leq1,$ then%
\begin{equation}
\left\Vert A\right\Vert _{\ast}\leq m\sqrt{n} \label{hin}%
\end{equation}

\end{proposition}

Inequality (\ref{hin0}) follows immediately by the AM-QM inequality, but the
important point here is which matrices force equality in (\ref{hin0}), and
which in (\ref{hin}). As seen in Theorem \ref{MCgen} below, a matrix $A$
forces equality in (\ref{hin0}) if and only if $AA^{\ast}=aI_{m}$ for some
$a>0$. Hence, $A$ forces equality in (\ref{hin}) if and only if $A$ is a
partial Hadamard matrix.\medskip

In short, among the complex $m\times n$ matrices with entries of modulus $1$,
partial Hadamard matrices are precisely those with trace norm equal to
$m\sqrt{n}$. In particular, if $n=m,$ these are the matrices with determinant
equal to $n^{n/2}$; however, the latter characterization is unwieldy and of
limited use, for if $n\neq m,$ determinants are nonexistent, whereas the trace
norm always fits the bill.\medskip

We shall give inequalities similar to (\ref{hin}) for various other norms.
Usually their simple proofs proceed by the PM inequality, but a constructive
description of the matrices that force equality may be quite hard, sometimes
even hopeless. Nevertheless, we always discuss the cases of equality, for they
are the most instructive bit.

\subsection{Structure of the survey}

The survey covers results of the papers \cite{GHK01}, \cite{KhRe08},
\cite{Nik11c}, \cite{Nik12}, \cite{Nik15b}, \cite{Nik15c}, and \cite{NiYu13},
with proofs of the known results suppressed; however, results appearing here
for the first time are proved in full.

Section \ref{secT} is dedicated to the trace norm, with heavy emphasis on
matrices. It starts with an introduction of Hadamard and conference matrices,
and continues with the maximum trace norm of $r$-partite graphs and matrices,
followed by Nordhaus-Gaddum problems about the trace norm. The section closes
with a collection of open problems about graphs with maximal energy.

Section \ref{secK} is about the Ky Fan norms, with emphasis on the maximal Ky
Fan norms of graphs. In this topic, a crucial role plays a class of matrices
extending real symmetric Hadamard matrices. We analyze this class and find the
maximal Ky Fan $k$-norm of graphs for infinitely many $k.$ We also survey
Nordhaus-Gaddum problems and show a few relations of the clique and the
chromatic numbers to Ky Fan norms.

Section \ref{secS} is dedicated to the Schatten norms. The section starts with
a few results on the Schatten $p$-norm as a function of $p.$ This direction is
totally new, with no roots in the study of graph energy. The other topics of
Section \ref{secS} are: graphs and matrices with extremal Schatten norms;
$r$-partite matrices and graphs with maximal Schatten norms;  Schatten norms
of trees; and Schatten norms of random graphs.

\section{\label{secT}The trace norm}

We begin this section with a discussion of Hadamard and conference matrices,
which are essential for bounds on the norms of graphs and matrices. In
particular, we stress on several analytic characterizations of these matrices
via their singular values, which are used throughout the survey.\medskip

In Section \ref{secGM}, we give bounds on the trace norms of rectangular
complex and nonnegative matrices conceived in the spirit of graph energy.
Particularly curious is Proposition \ref{p1}---a \textquotedblleft bound
generator\textquotedblright\ that can be used to convert lower bounds on the
operator norm into upper bounds on the trace norm.\medskip

In Section \ref{secMTN}, we present results on the maximal trace norm of
$r$-partite graphs and matrices, which recently appeared in \cite{Nik15a}.
This section brings conference matrices to the limelight; note that hitherto
these matrices have been barely used in graph energy.\medskip

In Section \ref{NGtr}, we outline solutions of some Nordhaus-Gaddum problems
about the trace norm of nonnegative matrices and describe a tight solution to
a problem of Gutman and Zhou. The section shows the exclusive role of the
conference graphs for Nordhaus-Gaddum problems about the trace norm.\medskip

Finally, in Section \ref{OPtr}, we raise a number of open problems about
graphs of maximal energy, which suggest that in many directions we are yet to
face the real difficulties in the study of graph energy.

\subsection{Hadamard and conference matrices}

Recall three well-known definitions:\emph{ \medskip}

An \emph{Hadamard matrix} of order $n$ is an $n\times n$ matrix $H,$ with
entries of modulus $1,$ such that $HH^{\ast}=nI_{n}.$

More generally, if $n\geq m,$ a \emph{partial Hadamard matrix} is an $m\times
n$ matrix $H,$ with entries of modulus $1,$ such that $HH^{\ast}=nI_{m}.$

A \emph{conference matrix} of order $n$ is an $n\times n$ matrix $C$ with zero
diagonal, with off-diagonal entries of modulus $1$, such that $CC^{\ast
}=\left(  n-1\right)  I_{n}.$\medskip

Unfortunately, the above definitions, albeit concise and sleek, fail to
emphasize the crucial role of the singular values of these matrices. To make
our point clear, we shall outline several characterizations via singular
values. Obviously, all singular values of an $n\times n$ Hadamard matrix are
equal to $\sqrt{n}$; also, assuming that $m\leq n,$ all nonzero singular
values of an $m\times n$ partial Hadamard matrix are equal to $\sqrt{n}.$
Likewise, all singular values of an $n\times n$ conference matrix are equal to
$\sqrt{n-1}.$

\begin{proposition}
\label{pro1}If $H:=\left[  h_{i,j}\right]  $ is an $n\times n$ complex matrix,
with $\left\vert h_{i,j}\right\vert =1$ for all $i,j\in\left[  n\right]  ,$
then the following conditions are equivalent:

(i) $H$ satisfies the equation $HH^{\ast}=nI_{n}$;

(ii) each singular value of $H$ is equal to $\sqrt{n}$;

(iii) $\left\Vert H\right\Vert _{\ast}=n\sqrt{n}$;

(iv) the largest singular value of $H$ is equal to $\sqrt{n}$;

(v) the smallest singular value of $H$ is equal to $\sqrt{n}.$
\end{proposition}

\begin{proof}
The implications (i) $\Rightarrow$ (ii) $\Rightarrow$ (iii), (ii)
$\Rightarrow$ (iv), and (ii) $\Rightarrow$ (v) are obvious. As in (\ref{in2}),
we see that $\left\Vert H\right\Vert _{\ast}\leq n\sqrt{n},$ with equality if
and only if $\sigma_{1}\left(  H\right)  =\cdots=\sigma_{k}\left(  H\right)
=\sqrt{n}.$ Hence, (iii) $\Rightarrow$ (ii).

To prove the implications (iv) $\Rightarrow$ (i) and (v) $\Rightarrow$ (i),
note that the diagonal elements of $HH^{\ast}$ are equal to $n,$ as these are
the inner product of the rows of $H$ with themselves. Hence, if $HH^{\ast
\text{ }}$contains a nonzero off-diagonal element $a$, then $HH^{\ast\text{ }%
}$contains a principal $2\times2$ matrix
\[
B=\left[
\begin{array}
[c]{cc}%
n & \overline{a}\\
a & n
\end{array}
\right]  .
\]
Clearly, the characteristic polynomial of $B$ is $x^{2}-nx+n^{2}-\left\vert
a\right\vert ^{2},$ and so
\[
\lambda_{1}\left(  B\right)  =n+\left\vert a\right\vert \text{ and }%
\lambda_{2}\left(  B\right)  =n-\left\vert a\right\vert .
\]
Now, Cauchy's interlacing theorem implies that
\[
\sigma_{1}^{2}\left(  H\right)  =\lambda_{1}\left(  HH^{\ast}\right)
\geq\lambda_{1}\left(  B\right)  >\lambda_{2}\left(  B\right)  \geq\lambda
_{n}\left(  HH^{\ast}\right)  =\sigma_{n}^{2}\left(  H\right)  .
\]
Therefore, if (i) fails, then both (iv) and (v) fail. Hence, the implications
(iv) $\Rightarrow$ (i) and (v) $\Rightarrow$ (i) hold as well, completing the
proof of Proposition \ref{pro1}.
\end{proof}

Proposition \ref{pro1} immediately implies the following essential bound on
the trace norm of complex square matrices.

\begin{proposition}
\label{tHad}If $A$ is an $n\times n$ matrix with $\left\Vert A\right\Vert
_{\max}\leq1,$ then%
\[
\left\Vert A\right\Vert _{\ast}\leq n\sqrt{n}.
\]
Equality holds if and only if all singular values of $A$ are equal to
$\sqrt{n}.$ Equivalently, equality holds if and only if $A$ is a Hadamard matrix.
\end{proposition}

The analytic characterizations (ii)-(v) in Proposition \ref{pro1} are by far
simpler and clearer than (i); in addition, we see that Hadamard matrices may
be introduced and studied regardless of determinants, which obscure the
picture in\ this case. This point of view is further supported by the seamless
extension of Propositions \ref{pro1} and \ref{tHad} to nonsquare partial
Hadamard matrices, where determinants simply do not exists. Here are the
corresponding statements:

\begin{proposition}
\label{pro1mn}Let $n\geq m\geq2.$ If $H:=\left[  h_{i,j}\right]  $ is an
$m\times n$ complex matrix, with $\left\vert h_{i,j}\right\vert =1$ for all
$i\in\left[  m\right]  ,$ $j\in\left[  n\right]  ,$ then the following
conditions are equivalent:

(i) $H$ satisfies the equation $HH^{\ast}=nI_{m}$;

(ii) each of the first $m$ singular values of $H$ is equal to $\sqrt{n}$;

(iii) $\left\Vert H\right\Vert _{\ast}=m\sqrt{n}$;

(iv) the largest singular value of $H$ is equal to $\sqrt{n}$;

(v) the smallest nonzero singular value of $H$ is equal to $\sqrt{n}.$
\end{proposition}

\begin{proposition}
\label{tpHad}Let $n\geq m\geq2.$ If $A$ is an $m\times n$ matrix with
$\left\Vert A\right\Vert _{\max}\leq1,$ then%
\[
\left\Vert A\right\Vert _{\ast}\leq m\sqrt{n}.
\]
Equality holds if and only if all nonzero singular values of $A$ are equal to
$\sqrt{n}.$ Equivalently, equality holds if and only if $A$ is a partial
Hadamard matrix.
\end{proposition}

Finally, Propositions \ref{pro1} and \ref{tHad} can be adapted for conference matrices:

\begin{proposition}
\label{proC}If $C$ is an $n\times n$ complex matrix with zero diagonal, with
off-diagonal entries of modulus $1,$ then the following conditions are equivalent:

(i) $C$ satisfies the equation $CC^{\ast}=\left(  n-1\right)  I_{n}$;

(ii) each singular value of $C$ is equal to $\sqrt{n-1}$;

(iii) $\left\Vert C\right\Vert _{\ast}=n\sqrt{n-1}$;

(iv) the largest singular value of $C$ is equal to $\sqrt{n-1}$;

(v) the smallest singular value of $C$ is equal to $\sqrt{n-1}.$
\end{proposition}

\begin{proposition}
\label{tcon}If $A$ is an $n\times n$ matrix with zero diagonal and $\left\Vert
A\right\Vert _{\max}\leq1,$ then%
\[
\left\Vert A\right\Vert _{\ast}\leq n\sqrt{n-1}.
\]
Equality holds if and only if all singular values of $A$ are equal to
$\sqrt{n-1}.$ Equivalently, equality holds if and only if $A$ is a conference matrix.
\end{proposition}

In summary: Hadamard, partial Hadamard, and conference matrices are unique
solutions of basic extremal analytic problems about the trace norm. Moreover,
the definitions of these matrices via singular values are simpler and clearer
than the standard definitions.

\subsection{\label{secGM}The trace norm of general matrices}

In this section we give several bounds on the trace norm of general matrices,
which were proved in \cite{Nik07i}. Most proofs are omitted here, but in
Section \ref{SSS} these results will be extended to Schatten norms and their
proofs will be outlined.\medskip

We start with a strengthening of bound (\ref{hin0}).

\begin{proposition}
\label{p1}If $n\geq m\geq2$ and $A\in M_{m,n},$ then%
\begin{equation}
\left\Vert A\right\Vert _{\ast}\leq\sigma_{1}\left(  A\right)  +\sqrt{\left(
m-1\right)  \left(  \left\Vert A\right\Vert _{2}^{2}-\sigma_{1}^{2}\left(
A\right)  \right)  }. \label{b1}%
\end{equation}
Equality holds in (\ref{b1}) if and only if $\sigma_{2}\left(  A\right)
=\cdots=\sigma_{m}\left(  A\right)  $.
\end{proposition}

Inequality (\ref{b1}) merits closer attention. First, together with the
Cauchy-Schwarz inequality, (\ref{b1}) implies the simple and straightforward
inequality (\ref{hin0}). What's more, in exchange for its clumsy form,
inequality (\ref{b1}) provides extended versatility. Indeed, note that the
function
\[
f\left(  x\right)  :=x+\left(  m-1\right)  ^{-1/2}\left(  \left\Vert
A\right\Vert _{2}^{2}-x^{2}\right)  ^{1/2}%
\]
is decreasing whenever $\left\Vert A\right\Vert _{2}/\sqrt{m}\leq
x\leq\left\Vert A\right\Vert _{2}.$ Since
\[
\left\Vert A\right\Vert _{2}/\sqrt{m}\leq\sigma_{1}\left(  A\right)
\leq\left\Vert A\right\Vert _{2},
\]
if a number $C$ satisfies
\[
\sigma_{1}\left(  A\right)  \geq C\geq\left\Vert A\right\Vert _{2}/\sqrt{m},
\]
then $f\left(  \sigma_{1}\left(  A\right)  \right)  \leq f\left(  C\right)  ,$
and so
\[
\left\Vert A\right\Vert _{\ast}\leq C+\sqrt{\left(  m-1\right)  \left(
\left\Vert A\right\Vert _{2}^{2}-C^{2}\right)  }.
\]
Therefore, Proposition \ref{p1} is a vehicle for converting lower bounds on
$\sigma_{1}\left(  A\right)  $ into upper bounds on $\left\Vert A\right\Vert
_{\ast}.$ Lower bounds on $\sigma_{1}\left(  A\right)  $ have been studied,
see, e.g., \cite{Nik07a} for an infinite family of such bounds, which, under
some restrictions may give an infinite family of upper bounds on $\left\Vert
A\right\Vert _{\ast}.$ In particular, note the following basic lower bounds on
$\sigma_{1}\left(  A\right)  :$

\begin{proposition}
\label{lobos}Let $A\in M_{m,n},$ let $r_{1},\ldots,r_{m}$ be the row sums of
$A,$ and let $c_{1},\ldots,c_{n}$ be its column sums. Then
\[
\sigma_{1}\left(  A\right)  \geq\sqrt{\frac{1}{n}\left(  \left\vert
r_{1}\right\vert ^{2}+\cdots+\left\vert r_{m}\right\vert ^{2}\right)  },
\]%
\[
\sigma_{1}\left(  A\right)  \geq\sqrt{\frac{1}{m}\left(  \left\vert
c_{1}\right\vert ^{2}+\cdots+\left\vert c_{n}\right\vert ^{2}\right)  },
\]
and%
\begin{equation}
\sigma_{1}\left(  A\right)  \geq\frac{1}{\sqrt{mn}}\sum\nolimits_{i}\left\vert
r_{i}\right\vert . \label{lbs}%
\end{equation}
If equality holds in (\ref{lbs}), then $\left\vert r_{1}\right\vert
=\cdots=\left\vert r_{m}\right\vert $. If $A$ is nonnegative, equality holds
in (\ref{lbs}) if and only if $A$ is regular.
\end{proposition}

Propositions \ref{lobos} leads to three bounds on $\left\Vert A\right\Vert
_{\ast}$; we spell out just one of them, generalizing a well-known result of
Koolen and Moulton \cite{KoMo01}.

\begin{proposition}
\label{pronm}Let $n\geq m\geq2$ and $A=\left[  a_{i,j}\right]  \in M_{m,n}.$
If%
\[
\frac{1}{\sqrt{mn}}\sum\nolimits_{i}\left\vert \sum\nolimits_{j}%
a_{ij}\right\vert \geq\sqrt{n}\left\Vert A\right\Vert _{2},
\]
then
\begin{equation}
\left\Vert A\right\Vert _{\ast}\leq\frac{1}{\sqrt{mn}}\sum\nolimits_{i}%
\left\vert \sum\nolimits_{j}a_{ij}\right\vert +\sqrt{\left(  m-1\right)
\left(  \left\Vert A\right\Vert _{2}^{2}-\frac{1}{mn}\left(  \sum
\nolimits_{i}\left\vert \sum\nolimits_{j}a_{ij}\right\vert \right)
^{2}\right)  }. \label{b2}%
\end{equation}
If equality holds in (\ref{b2}), then $\sigma_{2}\left(  A\right)
=\cdots=\sigma_{m}\left(  A\right)  ,$ the row sums of $A$ are equal in
absolute value and $\sigma_{1}\left(  A\right)  =\left(  mn\right)
^{-1/2}\sum\nolimits_{i}\left\vert \sum\nolimits_{j}a_{ij}\right\vert .$

If $A$ is nonnegative, then equality holds in (\ref{b2}) if and only if $A$ is
regular and $\sigma_{2}\left(  A\right)  =\cdots=\sigma_{m}\left(  A\right)  $.
\end{proposition}

The conditions for equality in (\ref{b1}) and (\ref{b2}) are stated here for
the first time, but we omit their easy proofs.\medskip\ 

Next, we recall an absolute bound on $\left\Vert A\right\Vert _{\ast}$ :

\begin{proposition}
\label{pronn}Let $n\geq m\geq2$ and let $A$ be an $m\times n$ nonnegative
matrix. If $\left\Vert A\right\Vert _{\max}\leq1,$ then
\begin{equation}
\left\Vert A\right\Vert _{\ast}\leq\frac{m\sqrt{n}}{2}+\frac{\sqrt{mn}}{2}.
\label{b3}%
\end{equation}
Equality holds if and only if the matrix $2A-J_{m,n}$ is a regular partial
Hadamard matrix.
\end{proposition}

Obviously, Proposition \ref{pronn} generalizes Koolen and Moulton's result
(\ref{KM}); inequality (\ref{b3}) is proved in \cite{Nik07i}, but the
characterization of the equality is new. In turn, Proposition \ref{pronn} is
generalized in Theorems \ref{tNik} and \ref{KMSmn} below.

Let us note that for $\left(  0,1\right)  $-matrices Kharaghani and
Tayfeh-Rezaie \cite{KhRe08} have characterized the conditions for equality in
inequality (\ref{b3}), using a different wording:

\begin{theorem}
[\cite{KhRe08}]\label{ThKTR}Let $n\geq m\geq2.$ If $A$ is an $m\times n$
$\left(  0,1\right)  $-matrix, then%
\[
\left\Vert A\right\Vert _{\ast}\leq\frac{m\sqrt{n}}{2}+\frac{\sqrt{mn}}{2}.
\]
Equality holds if and only if $A$ is the incidence matrix of a balanced
incomplete block design with parameters%
\[
m,n,n(m+\sqrt{m})/2m,(m+\sqrt{m})/2,n(m+2\sqrt{m})/4m.
\]

\end{theorem}

Observe that Proposition \ref{pronn} and Theorem \ref{ThKTR} have direct
consequences for bipartite graphs. Indeed, the adjacency matrix $A$ of a
bipartite graph $G$ can be written as a block matrix%
\[
A=\left[
\begin{array}
[c]{cc}%
0 & B^{T}\\
B & 0
\end{array}
\right]  ,
\]
for some $\left(  0,1\right)  $-matrix $B$, which is called the\emph{
biadjacency matrix} of $G.$

Since every $\left(  0,1\right)  $-matrix is the biadjacency matrix of a
certain bipartite graph, in a sense, the study of bipartite graphs is
equivalent to the study of rectangular $\left(  0,1\right)  $-matrices. In
particular, it is known that the spectrum of $G$ consists of the nonzero
singular values of $B,$ together with their negatives and possibly some zeros;
thus, if $\left\Vert \cdot\right\Vert $ is a Ky Fan or a Schatten norm, then
\[
\left\Vert G\right\Vert =2\left\Vert B\right\Vert .
\]
Hence, studying norms of bipartite graphs is tantamount to studying norms of
arbitrary $\left(  0,1\right)  $-matrices. For instance, Proposition
\ref{pronn} extends the following result of Koolen and Moulton about bipartite
graphs \cite{KoMo03}:\medskip

\emph{If }$G$\emph{ is a bipartite graph of order }$n,$\emph{ then}%
\begin{equation}
\left\Vert G\right\Vert _{\ast}\leq\frac{n\sqrt{n}}{2\sqrt{2}}+\frac{n}{2},
\label{KMb}%
\end{equation}
\emph{ with equality if and only if }$G$\emph{ is the incidence graph of a
design of specific type.}\medskip

Finally, let us mention one lower bound on the trace norm of general matrices.
Note that for any matrix $A,$ $\sigma_{2}\left(  A\right)  \neq0$ if and only
if the rank of $A$ is at least $2.$ For matrices of rank at least $2,$ the
following lower bound was given in \cite{Nik07i}.

\begin{proposition}
\label{prol}If the rank of a matrix $A=\left[  a_{i,j}\right]  $ is at least
$2,$ then%
\begin{equation}
\left\Vert A\right\Vert _{\ast}\geq\sigma_{1}\left(  A\right)  +\frac
{1}{\sigma_{2}\left(  A\right)  }\left(  \sum\nolimits_{i,j}\left\vert
a_{ij}\right\vert ^{2}-\sigma_{1}^{2}\left(  A\right)  \right)  . \label{lobo}%
\end{equation}
Equality holds if and only if all nonzero eigenvalues of $G$ other than
$\lambda$ have the same absolute value.
\end{proposition}

Bound (\ref{lobo}) is quite efficient for graphs: for example, equality holds
in (\ref{lobo}) if $A$ is the adjacency matrix of a design graph, or a
complete graph, or a complete bipartite graph; however, there are also other
graphs whose adjacency matrix forces equality in (\ref{lobo}).

\begin{problem}
Give a constructive characterization of all graphs $G$ such that the nonzero
eigenvalues of $G$ other than its largest eigenvalue have the same absolute value.
\end{problem}

\subsection{\label{secMTN}The trace norm of $r$-partite graphs and matrices}

In the previous section we saw that the inequality (\ref{KMb}) of Koolen and
Moulton can be extended using the biadjacency matrix of a bipartite graph.
Inequality (\ref{KMb}) may also be extended in another direction, namely to
$r$-partite graphs for $r\geq3$.

Recall that a graph is called $r$\emph{-partite}\ if its vertices can be
partitioned into $r$ edgeless sets. Clearly, inequality (\ref{KMb}) leads to
the following natural problem:

\begin{problem}
\label{pr}What is the maximum trace norm of an $r$-partite graph of order $n$?
\end{problem}

For complete $r$-partite graphs the question was answered in \cite{CCGH99},
but in general, Problem \ref{pr} is much more difficult, as, for almost all
$r$ and $n,$ it requires constructions that presently are beyond reach.
Nonetheless, some tight approximate results were obtained recently in the
paper \cite{Nik15a} and in its ArXiv version. Below, we give an outline of
these results, with their proofs suppressed. In Section \ref{secS}, we prove
somewhat stronger results for the Schatten norms.\medskip

First, we generalize the notion of $r$-partite graph to complex square
matrices. To this end, given an $n\times n$ matrix $A=\left[  a_{i,j}\right]
$ and nonempty sets $I\subset\left[  n\right]  ,$ $J\subset\left[  n\right]
,$ let us write $A\left[  I,J\right]  $ for the submatrix of all $a_{i,j}$
with $i\in I$ and $j\in J.$ Now, $r$-partite matrices are defined as follows:\ 

\begin{definition}
An $n\times n$ matrix $A$ is called $r$\textbf{-partite} if there is a
partition of its index set $\left[  n\right]  =N_{1}\cup\cdots\cup N_{r}$ such
that $A\left[  N_{i},N_{i}\right]  =0$ for any $i\in\left[  r\right]  $.
\end{definition}

Clearly, the adjacency matrix of an $r$-partite graph is an $r$-partite
matrix, but our definition extends to any square matrix.\medskip

We continue with an upper bound on the trace norm of $r$-partite matrices. For
a proof see the more general Theorem \ref{ScMxr}.

\begin{theorem}
\label{thMxr}Let $n\geq r\geq2,$ and let $A$ be an $n\times n$ matrix with
$\left\Vert A\right\Vert _{\max}\leq1.$ If $A$ is $r$-partite, then
\begin{equation}
\left\Vert A\right\Vert _{\ast}\leq n^{3/2}\sqrt{1-1/r}.\label{Mb}%
\end{equation}
Equality holds if and only if all singular values of $A$ are equal to
$\sqrt{\left(  1-1/r\right)  n}$.
\end{theorem}

The most challenging point of Theorem \ref{thMxr} is the case of equality in
(\ref{Mb}). Any matrix $A=\left[  a_{i,j}\right]  $ that forces equality in
(\ref{Mb}) has a long list of special properties, e.g.:

- $r$ divides $n;$

- all its partition sets are of size $n/r$;

- if an entry $a_{i,j}$ is not in a diagonal block, then $\left\vert
a_{i,j}\right\vert =1;$

- $AA^{\ast}=\left(  1-1/r\right)  nI_{n}.$

We see that the rows of $A$ are orthogonal, and so are its columns. Despite
these many necessary conditions, it seems hard to find for which $r$ and $n$
such matrices exist.

\begin{problem}
Give a constructive characterization of all $r$-partite $n\times n$ matrices
$A$ with $\left\Vert A\right\Vert _{\max}\leq1$ such that
\[
\left\Vert A\right\Vert _{\ast}=n^{3/2}\sqrt{1-1/r}.
\]

\end{problem}

We cannot solve this difficult problem in general, but nevertheless, we can
show that if $r$ is the order of a conference matrix, then equality holds in
(\ref{Mb}) for infinitely many $r$-partite matrices.

\begin{theorem}
\label{th3}Let $r$ be the order of a conference matrix, and let $k$ be the
order of an Hadamard matrix. There exists an $r$-partite matrix $A$ of order
$n=rk$ with $\left\Vert A\right\Vert _{\max}=1$ such that
\[
\left\Vert A\right\Vert _{\ast}=n^{3/2}\sqrt{1-1/r}.
\]

\end{theorem}

For a proof of Theorem \ref{th3} see the more general Theorem \ref{Scth3}%
.\medskip

Complex Hadamard matrices of order $n$ exists for any $n.$ This is not true
for conference matrices and real Hadamard matrices, although numerous
constructions of such matrices are known, e.g., Paley's constructions, which
are as follows:\medskip

\emph{If }$q$\emph{ is an odd prime power, then:}

- \emph{there is a real conference matrix of order }$q+1,$\emph{ which is
symmetric if }$q=1$\emph{ }$\left(  \operatorname{mod}4\right)  $;

- \emph{there is a real Hadamard matrix of order }$q+1$\emph{ if }$q=3$\emph{
}$\left(  \operatorname{mod}4\right)  $;

- \emph{there is a real, symmetric Hadamard matrix of order }$2\left(
q+1\right)  $\emph{ if }$q=1$\emph{ }$\left(  \operatorname{mod}4\right)
.$\medskip

It is known that there is no conference matrix of order $3,$ so we have an
intriguing problem:

\begin{problem}
Let $f\left(  n\right)  $ be the maximal trace norm of a tripartite matrix of
order $n,$ with $\left\Vert A\right\Vert _{\max}\leq1.$ Find%
\[
\lim_{n\rightarrow\infty}f\left(  n\right)  n^{-3/2}.
\]

\end{problem}

The existence of the above limit was proved in the ArXiv version of
\cite{Nik15a}.\medskip

Next, we give an approximate solution of Problem \ref{pr}. We state a bound
for nonnegative matrices, which is valid also for graphs.

\begin{theorem}
\label{th2}Let $n\geq r\geq2$, and let $A$ be an $n\times n$ nonnegative
matrix with $\left\Vert A\right\Vert _{\max}\leq1.$ If $A$ is $r$-partite,
then
\begin{equation}
\left\Vert A\right\Vert _{\ast}\leq\frac{n^{3/2}}{2}\sqrt{1-1/r}+\left(
1-1/r\right)  n. \label{nnb}%
\end{equation}

\end{theorem}

For a proof of Theorem \ref{th2} see the more general Theorem \ref{Sth2}%
.\medskip

Note that the matrix $A$ in Theorems \ref{thMxr} and \ref{th2} needs not be
symmetric; nonetheless, the following immediate corollary implies precisely
Koolen and Moulton's bound (\ref{KMb}) for bipartite graphs ($r=2$).

\begin{corollary}
\label{cor1}Let $n\geq r\geq2.$ If $G\ $is an $r$-partite graph of order $n,$
then
\begin{equation}
\left\Vert G\right\Vert _{\ast}\leq\frac{n^{3/2}}{2}\sqrt{1-1/r}+\left(
1-1/r\right)  n. \label{eab}%
\end{equation}

\end{corollary}

The linear in $n$ term in the right sides of bounds (\ref{nnb}) and
(\ref{eab}) can be diminished for $r\geq3$ by more involved methods; see the
ArXiv version of \cite{Nik15a}. However, the improved bounds in \cite{Nik15a}
are quite complicated and leave no hope for expressions as simple as in
(\ref{KMb}).\medskip

Further, the construction in Theorem \ref{th3} can be modified to provide
matching lower bounds for Theorem \ref{th2} and Corollary \ref{cor1}. For a
proof of Theorem \ref{th4} see Theorem \ref{Sth4}.

\begin{theorem}
\label{th4}Let $r$ be the order of a real symmetric conference matrix. If $k$
is the order of a real symmetric Hadamard matrix, then there is an $r$-partite
graph $G$ of order $n=rk$ with
\begin{equation}
\left\Vert G\right\Vert _{\ast}\geq\frac{n^{3/2}}{2}\sqrt{1-1/r}-\left(
1-1/r\right)  n. \label{eabl}%
\end{equation}

\end{theorem}

Note that bounds (\ref{eab}) and (\ref{eabl}) differ only in their linear
terms; however, getting rid of this difference seems very hard, and needs
improvements in both (\ref{eab}) and (\ref{eabl}).

\subsection{\label{NGtr}Nordhaus-Gaddum problems for the trace norm}

Let $G\ $be a graph $G$ of order $n.$ It seems not widely known that the
energy of the complement $\overline{G}$ of $G$ is not too different from the
energy of $G$. Indeed, if $A$ and $\overline{A}$ are the adjacency matrices of
$G$ and $\overline{G},$ then $A+\overline{A}=J_{n}-I_{n}$ and using the
triangle inequality for the trace norm, we find that%
\[
\left\Vert \overline{A}\right\Vert _{\ast}=\left\Vert J_{n}-I_{n}-A\right\Vert
_{\ast}\leq\left\Vert A\right\Vert _{\ast}+\left\Vert J_{n}-I_{n}\right\Vert
_{\ast}=\left\Vert G\right\Vert _{\ast}+2n-2.
\]
By symmetry, we get the following proposition:

\begin{proposition}
If $G$ is a graph of order $n$ and $\overline{G}$ is the complement of $G$,
then
\begin{equation}
\left\vert \left\Vert \overline{G}\right\Vert _{\ast}-\left\Vert G\right\Vert
_{\ast}\right\vert \leq2n-4. \label{cin}%
\end{equation}
Equality in (\ref{cin}) holds if and only if $G$ or $\overline{G}$ is a
complete graph.
\end{proposition}

Inequality (\ref{cin}) can be made more precise using Weyl's inequalities for
the eigenvalues of Hermitian matrices:

\begin{proposition}
\label{pro6}If $G$ is a graph of order $n$ and $\overline{G}$ is the
complement of $G$, then
\[
\left\Vert G\right\Vert _{\ast}-\left\Vert \overline{G}\right\Vert _{\ast}%
\leq2\lambda_{1}(G)
\]
and
\[
\left\Vert \overline{G}\right\Vert _{\ast}-\left\Vert G\right\Vert _{\ast}%
\leq2\lambda_{1}(\overline{G}).
\]

\end{proposition}

It seems that the bounds in Proposition \ref{pro6} can be improved, so we
raise the following problem:

\begin{problem}
Find the best possible upper bounds for $\left\Vert G\right\Vert _{\ast
}-\left\Vert \overline{G}\right\Vert _{\ast}$ for general and for regular graphs.
\end{problem}

As shown by Koolen and Moulton in \cite{KoMo01}, $G$ satisfies $\left\Vert
G\right\Vert _{\ast}=n\sqrt{n}/2+n/2$ if and only if $G\ $is a strongly
regular graph with parameters%
\begin{equation}
\left(  n,\left(  n+\sqrt{n}\right)  /2,\left(  n+2\sqrt{n}\right)  /4,\left(
n+2\sqrt{n}\right)  /4\right)  . \label{Gpar}%
\end{equation}
Hence, if $\left\Vert G\right\Vert _{\ast}=n\sqrt{n}/2+n/2,$ it is not hard to
see that the complement $\overline{G}$ satisfies $\left\Vert \overline
{G}\right\Vert _{\ast}<n\sqrt{n}/2+n/2$. This observation led Gutman and Zhou
\cite{ZhGu07} to the following natural problem:

\begin{problem}
\label{GZ}What is the maximum $\mathcal{E}\left(  n\right)  $ of the sum
$\left\Vert G\right\Vert _{\ast}+\left\Vert \overline{G}\right\Vert _{\ast},$
where $G$ is a graph of order $n?$
\end{problem}

In \cite{ZhGu07}, Gutman and Zhou proved a tight upper bound on $\mathcal{E}%
\left(  n\right)  :$
\begin{equation}
\mathcal{E}\left(  n\right)  \leq\left(  n-1\right)  \sqrt{n-1}+\sqrt{2}n.
\label{GZ1}%
\end{equation}

Clearly, Problem \ref{GZ} is a Nordhaus-Gaddum problem for the trace norm of
graphs. Before discussing Problem \ref{GZ} further, we introduce conference
and Paley graphs:

\begin{definition}
A \textbf{conference graph} of order $n$ is a strongly regular graph with
parameters%
\[
\left(  n,\left(  n-1\right)  /2,\left(  n-5\right)  /4,\left(  n-1\right)
/4\right)  .
\]

\end{definition}

It is easy to see that the eigenvalues of a conference graph of order $n$ are
\[
\left(  n-1\right)  /2,\left(  (\sqrt{n}-1)/2\right)  ^{\left[  \left(
n-1\right)  /2\right]  },\left(  -\left(  \sqrt{n}+1\right)  /2\right)
^{\left[  \left(  n-1\right)  /2\right]  },
\]
where the numbers in brackets denote multiplicities. Note that the complement
of a conference graph is also a conference graph. The best known examples of
conference graphs are the Paley graphs $P_{q},$ which are defined as
follows:\medskip

\emph{Given a prime power }$q=1$\emph{ }$(\operatorname{mod}$\emph{ }%
$4),$\emph{ the vertices of }$P_{q}$\emph{ are the numbers }$1,\ldots,q$\emph{
and two vertices }$u,v$\emph{ are adjacent if }$\left\vert u-v\right\vert
$\emph{ is an exact square }$\operatorname{mod}$\emph{ }$q.$\medskip

Additional introductory and reference material on conference and Paley graphs
can be found in \cite{GoRo01}.\medskip

Returning to Problem \ref{GZ}, note that any conference graph of order $n$
provides the lower bound%
\begin{equation}
\mathcal{E}\left(  n\right)  \geq\left(  n-1\right)  \sqrt{n}+n-1,
\label{maxen1}%
\end{equation}
which matches the upper bound of Gutman and Zhou (\ref{GZ1}) up to a linear term.

In \cite{NiYu13}, it was shown that conference graphs are, in fact, extremal
for Problem \ref{GZ}, as shown in the following more general matrix statement:

\begin{theorem}
\label{NG1}Let $n\geq7,$ and let $A$ be an $n\times n$ symmetric nonnegative
matrix with zero diagonal. If $\left\Vert A\right\Vert _{\max}\leq1$,, then
\begin{equation}
\left\Vert A\right\Vert _{\ast}+\left\Vert J_{n}-I_{n}-A\right\Vert _{\ast
}\leq\left(  n-1\right)  \sqrt{n}+n-1, \label{main1}%
\end{equation}
with equality holding if and only if $A$ is the adjacency matrix of a
conference graph.
\end{theorem}

The proof of Theorem \ref{NG1} is not easy and involves some new analytic and
combinatorial techniques using Weyl's inequalities for sums of Hermitian
matrices. It is based on the following two results of separate interest.

\begin{theorem}
\label{NG2}Let $A$ be an $n\times n$ nonnegative matrix with zero diagonal. If
$\left\Vert A\right\Vert _{\max}\leq1$, then
\begin{equation}
\left\Vert A+\frac{1}{2}I_{n}\right\Vert _{\ast}+\left\Vert J_{n}-A-\frac
{1}{2}I_{n}\right\Vert _{\ast}\leq(n-1)\sqrt{n}+n. \label{th2in}%
\end{equation}
Equality holds if and only if $A$ is a $\left(  0,1\right)  $-matrix, with all
row and column sums equal to $\left(  n-1\right)  /2,$ and with $\sigma
_{i}\left(  A+\frac{1}{2}I_{n}\right)  =\sqrt{n}/2$ for every $i=2,\ldots,n.$
\end{theorem}

\begin{corollary}
\label{cor2}Let $A$ be an $n\times n$ symmetric nonnegative matrix with zero
diagonal and with $\left\Vert A\right\Vert _{\max}\leq1.$ If
\begin{equation}
\left\Vert A+\frac{1}{2}I_{n}\right\Vert _{\ast}+\left\Vert J_{n}-A-\frac
{1}{2}I_{n}\right\Vert _{\ast}=(n-1)\sqrt{n}+n, \label{cor2in}%
\end{equation}
then $A$ is the adjacency matrix of a conference graph.
\end{corollary}

Let us note that the difficulty of the proof of Theorem \ref{NG1} stems from
the stipulations that $A$ is symmetric and its diagonal is zero. In Section
\ref{secK}, Theorem \ref{NG3}, we shall see that if these constraints are
omitted, one can prove a sweeping generalization for any Ky Fan norm, but it
is not tight for graphs. For convenience, here we state that result for the
trace norm:

\begin{theorem}
\label{NGtt}If $n\geq m$ and $A$ is an $m\times n$ nonnegative matrix with
$\left\Vert A\right\Vert _{\max}\leq1,$ then
\begin{equation}
\left\Vert A\right\Vert _{\ast}+\left\Vert J_{m,n}-A\right\Vert _{\ast}%
\leq\sqrt{m\left(  m-1\right)  n}+\sqrt{mn}. \label{NGt}%
\end{equation}
Equality holds in (\ref{NGt}) if only if $A$ is a $\left(  0,1\right)
$-matrix with
\[
\sigma_{1}(A)=\sigma_{1}(\overline{A})=\sqrt{mn}/2
\]
and
\[
\sigma_{2}(A)=\cdots=\sigma_{m}(A)=\sigma_{2}(\overline{A})=\cdots=\sigma
_{m}(\overline{A})=\frac{1}{2}\sqrt{\frac{mn}{m-1}}.
\]
Equivalently, equality holds in (\ref{NGt}) if only if the matrix
$H:=2A-J_{m,n}$ has zero row sums and column sums and
\[
\sigma_{1}(H)=\cdots=\sigma_{m-1}(H)=\sqrt{\frac{mn}{m-1}}.
\]

\end{theorem}

\medskip

\subsection{\label{OPtr}Some open problems on graph energy}

Bounds (\ref{KM}), (\ref{KMb}), Proposition \ref{pronm}, Theorem \ref{ThKTR}
may leave the false impression that the problems about graphs of maximal
energy are essentially solved. This is far from being true. Indeed, in all
these cases we have concise upper bounds, together with descriptions of sparse
sets of graphs for which these bounds are attained. Unfortunately these
descriptions are non-constructive and hinge on the unclear existence of
combinatorial objects like Hadamard matrices, strongly regular graphs, or
BIBDs. What's more, even if we knew all about these special cases of equality,
we still know nothing about the general case, which may considerably deviate
from the special cases.

Obviously, in such problems we have to follow the customary path of extremal
graph theory: that is to say, we have to define and investigate a particular
extremal function. For instance, for the energy, we have to define the
function%
\[
F_{\ast}\left(  n\right)  :=\max\left\{  \left\Vert G\right\Vert _{\ast
}:G\text{ is a graph of order }n\right\}  ,
\]
and come to grips with the following problem:

\begin{problem}
Find or approximate $F_{\ast}\left(  n\right)  $ for every $n.$
\end{problem}

To solve this problem we must give upper and lower bounds on $F_{\ast}\left(
n\right)  ,$ aiming to narrow the gap between them as much as possible. Note
that the closing of these bounds may go in rounds for decades.

Let us note that the lower bounds on $F_{\ast}\left(  n\right)  $ and on
similar extremal functions are usually based on constructions. For example, a
simple construction using the Paley graphs \cite{Nik07j} shows that
\[
F_{\ast}\left(  n\right)  \geq\frac{n\sqrt{n}}{2}-n^{11/10}.
\]
Improving this bound significantly is a major problem, mostly because of its
relation to possible orders of Hadamard matrices. Perhaps $F_{\ast}\left(
n\right)  \geq n^{3/2}/2$ for sufficiently large $n$, but this is far from
clear.\medskip

Let us state a few more problems of similar type.

\begin{problem}
\label{prb2}For every $n,$ find or approximate the function
\[
\max\left\{  \left\Vert G\right\Vert _{\ast}:G\text{ is a bipartite graph of
order }n\right\}  .
\]

\end{problem}

In fact, Problem \ref{prb2} can be extended to a two-parameter version, for
which results on partial Hadamard matrices \cite{Lau00} may provide some solutions:

\begin{problem}
Let $q\geq p\geq1.$ Find or approximate the function
\[
\max\left\{  \left\Vert G\right\Vert _{\ast}:G\text{ is a bipartite graph with
vertex classes of sizes }p\text{ and }q\right\}  .
\]

\end{problem}

At that point, Problem \ref{pr} comes in mind, but we shall not restate it
again.\ Instead, as it is unlikely that it will be solved satisfactorily in
the nearest future, we state the following simplified version of it:

\begin{problem}
Let $f_{r}\left(  n\right)  $ be the maximal trace norm of an $r$-partite
graph of order $n.$ Find%
\[
\lim_{n\rightarrow\infty}f_{r}\left(  n\right)  n^{-3/2}.
\]

\end{problem}

Note that the existence of the above limit was proved in the ArXiv version of
\cite{Nik15a}.\medskip

Next, in the spirit of the famous Tur\'{a}n theorem \cite{Tur41}, we raise the
following, probably difficult, problem:

\begin{problem}
Let $r\geq2.$ Find or approximate the function
\[
g_{r}\left(  n\right)  =\max\left\{  \left\Vert G\right\Vert _{\ast}:G\text{
is a graph with no complete subgraph of order }r+1\right\}  .
\]

\end{problem}

The author knows nothing even for $g_{2}\left(  n\right)  $. Since finding
precisely $g_{r}\left(  n\right)  $ for any $r\geq2$ seems difficult, it is
worth to consider the following simpler question:

\begin{problem}
Does the limit
\[
\gamma_{r}=\lim_{n\rightarrow\infty}g_{r}\left(  n\right)  n^{-3/2}%
\]
exist? If yes, find $\gamma_{r}.$
\end{problem}

Finally, we shall discuss two problems with a different setup. Let $G$ be a
graph of order $n$ with $m$ edges. In \cite{KoMo01}, Koolen and Moulton showed
that if $m\geq n/2,$ then
\begin{equation}
\left\Vert G\right\Vert _{\ast}\leq2m/n+\sqrt{\left(  n-1\right)  \left(
2m-\left(  2m/n\right)  ^{2}\right)  }, \label{KM0}%
\end{equation}
with equality holding if and only if $G=\left(  n/2\right)  K_{2},$ or
$G=K_{n},$ or $G$ is a strongly regular graph with parameters $\left(
n,k,a,a\right)  ,$ where
\[
k=2m/n,\text{ \ \ and \ \ }a=\left(  k^{2}-k\right)  /\left(  n-1\right)  .
\]
Recall the following definition (see, e.g., \ \cite{BJK99}, p. 144):

\begin{definition}
A strongly regular graph with parameters $\left(  n,k,a,a\right)  $ is called
a \textbf{design graph}.
\end{definition}

Design graphs are quite rare, and since $\left(  k^{2}-k\right)  /\left(
n-1\right)  =a\geq1$, for every $k,$ there are finitely many design graphs of
order $n$ and degree $k.$ Therefore, if $m$ is a slowly growing function of
$n,$ equality in (\ref{KM0}) does not hold if $n$ is sufficiently large. This
fact leads to the following problem:

\begin{problem}
Let $C\geq1.$ For all sufficiently large $n,$ find the maximum $\left\Vert
G\right\Vert _{\ast}$ if $G$ is a graph of order $n$ with at most $Cn$ edges.
\end{problem}

Further, recall that Koolen and Moulton deduced inequality (\ref{KM0}) from a
more general statement: \emph{If }$\lambda$\emph{ is the largest eigenvalue of
}$G,$\emph{ then}%
\begin{equation}
\left\Vert G\right\Vert _{\ast}\leq\lambda+\sqrt{\left(  n-1\right)  \left(
2m-\lambda^{2}\right)  }, \label{KM1}%
\end{equation}
\emph{with equality if and only if }$\sigma_{2}\left(  G\right)
=\cdots=\sigma_{n}\left(  G\right)  .$

The condition $\sigma_{2}\left(  G\right)  =\cdots=\sigma_{n}\left(  G\right)
$ is quite strong, but is hard to restate in non-spectral graph terms.

Suppose that a graph $G\ of$ order $n$ satisfies the condition $\sigma
_{2}\left(  G\right)  =\cdots=\sigma_{n}\left(  G\right)  .$ Clearly the
eigenvalues $\lambda_{2}\left(  G\right)  ,\ldots,\lambda_{n}\left(  G\right)
$ take only two values, and so $G$ is a graph with at most three eigenvalues.
If $G$ is regular, then either $G=\left(  n/2\right)  K_{2},$ or $G=K_{n},$ or
$G$ is a design graph.

If $G$ is not regular and is disconnected, then $G=K_{n-2r}+rK_{2}.$ We thus
arrive at the following problem:

\begin{problem}
Give a constructive characterization of all connected irregular graphs $G$ of
order $n$ with $\left\vert \lambda_{2}\left(  G\right)  \right\vert
=\cdots=\left\vert \lambda_{n}\left(  G\right)  \right\vert .$
\end{problem}

Given the problems listed above, one might predict that the real difficulties
in the study of graph energy are still ahead of us.

\section{\label{secK}The Ky Fan norms}

In this section we survey some results on the Ky Fan norms of graphs and
matrices, given in \cite{GHK01}, \cite{Nik11c}, \cite{Nik12}, \cite{Nik15b},
and \cite{Nik15c}. The main topic we are interested in is the maximal Ky Fan
$k$-norm of graphs of given order. More precisely, define the function
$\xi_{k}\left(  n\right)  $ as%
\[
\xi_{k}\left(  n\right)  :=\max\left\{  \left\Vert G\right\Vert _{\left[
k\right]  }:G\text{ is a graph of order }n\right\}  ,
\]
and consider the following natural problem:

\begin{problem}
\label{pKF}For all $n\geq k\geq2,$ find or approximate $\xi_{k}\left(
n\right)  .$
\end{problem}

To begin with, following the approach of \cite{Nik06},\ it is not hard to find
the asymptotics of $\xi_{k}\left(  n\right)  $:

\begin{proposition}
For every fixed positive integer $k,$ the limit $\xi_{k}=\lim
\limits_{n\rightarrow\infty}\xi_{k}\left(  n\right)  /n$ exists.
\end{proposition}

Note that the maximal energy, i.e., the maximal Ky Fan $n$-norm, is of order
$n^{3/2},$ whereas, for a fixed positive integer $k,$ the maximal Ky Fan
$k$-norm of an $n$ vertex graph is linear in $n.$ However, this fact does not
make the solution of Problem \ref{pKF} any easier.

It turns out that finding $\xi_{k}\left(  n\right)  $ is hard for any
$k\geq2,$ and even finding $\xi_{k}$ is challenging; in particular, $\xi_{2}$
is not known yet, despite intensive research. Recall that in \cite{GHK01},
Gregory, Hershkowitz and Kirkland asked what is the maximal value of the
spread of a graph of order $n,$ that is to say, what is
\[
\max_{v\left(  G\right)  =n}\lambda_{1}\left(  G\right)  -\lambda_{n}\left(
G\right)  .
\]
This problem is still open, and even an asymptotic solution is not known, but
in \cite{Nik11c} it was shown that finding the maximum spread of a graph of
order $n$ is equivalent to finding $\xi_{2}\left(  n\right)  $.\medskip

It turns out that the study of maximal Ky Fan norms of matrices yields new
insights into Hadamard matrices and partial Hadamard matrices. In Section
\ref{sec 1} we give several matrix results, and deduce an upper bound on
$\xi_{k}\left(  n\right)  $.

To find matching lower bounds, in Section \ref{sec eH} we discuss a class of
matrices, which have been introduced in \cite{Nik15b}. These matrices
generalize symmetric Hadamard matrices, and\ provide infinite families of
exact and approximate solutions to Problem \ref{pKF}, presented in Section
\ref{sec 2}.

In Section \ref{FD} we study Nordhaus-Gaddum problems for the Ky Fan norms of
graphs and matrices, and give several tight bounds.

Finally, in Section \ref{KFch}, we study relations of Ky Fan norms to the
chromatic number and the clique number of graphs.

\subsection{\label{sec 1}Maximal Ky Fan norms of matrices}

Applying the AM-QM inequality to the sum of the $k$ largest singular values
and using (\ref{meq}), we obtain the following theorem:

\begin{theorem}
\label{mo1}Let $n\geq m\geq2$ and $m\geq k\geq1.$ If $A\in M_{m,n},$ then
\begin{equation}
\left\Vert A\right\Vert _{\left[  k\right]  }\leq\sqrt{k}\left\Vert
A\right\Vert _{2}. \label{mkf}%
\end{equation}
Equality holds if and only if $A$ has exactly $k$ nonzero singular values,
which are equal.
\end{theorem}

It seems difficult to give a constructive characterization of all matrices
that force equality in (\ref{mkf}), since the given condition is exact, but is
too general for constructive characterization. Thus, we give just one
construction, showing the great diversity of this class:

Let $q\geq k$ and let $B$ be a $k\times q$ matrix whose rows are pairwise
orthogonal vectors of $l_{2}$-norm equal to $l$. Since $BB^{\ast}=l^{2}I_{k}$,
we see that all $k$ singular values of $B$ are equal to $l$. Now, choose
arbitrary $r\geq1$ and $s\geq1,$ and set $A:=B\otimes J_{r,s}.$ Obviously, $A$
has exactly $k$ nonzero singular values, which are equal, and so $\left\Vert
A\right\Vert _{\left[  k\right]  }=\sqrt{k}\left\Vert A\right\Vert
_{2}.\medskip$

Further, the inequality $\left\Vert A\right\Vert _{2}\leq\sqrt{mn}\left\Vert
A\right\Vert _{\max}$ implies the following corollary:

\begin{corollary}
\label{mo2}Let $n\geq m\geq k\geq2$ and $A\in M_{m,n}.$ If $\left\Vert
A\right\Vert _{\max}\leq1,$ then
\begin{equation}
\left\Vert A\right\Vert _{\left[  k\right]  }\leq\sqrt{kmn}. \label{Hadin}%
\end{equation}
Equality holds in (\ref{Hadin}) if and only if all entries of $A$ have modulus
$1,$ and $A$ has exactly $k$ nonzero singular values, which are equal to
$\sqrt{mn/k}$.
\end{corollary}

Here is a general construction of matrices that force equality in
(\ref{Hadin}): choose arbitrary $r\geq1,$ $s\geq1,$ $q\geq k,$ let $B$ be a
$k\times q$ partial Hadamard matrix, and set $A:=B\otimes J_{r,s}.$ Since $A$
has exactly $k$ nonzero singular values, and they are equal, and since the
entries of $A$ are of modulus $1$, we have $\left\Vert A\right\Vert _{\left[
k\right]  }=\sqrt{k}\left\Vert A\right\Vert _{2}=\sqrt{kmn},$ thus $A$ forces
equality in (\ref{Hadin}).\medskip

Although the upper bounds given in Theorem \ref{mo1} and Corollary \ref{mo2}
are as good as one can get, for nonnegative matrices they can be improved.

\begin{theorem}
\label{tNik}Let $n\geq m\geq k\geq2.$ If $A$ is an $m\times n$ nonnegative
matrix with $\left\Vert A\right\Vert _{\max}\leq1,$ then%
\begin{equation}
\left\Vert A\right\Vert _{\left[  k\right]  }\leq\frac{\sqrt{kmn}}{2}%
+\frac{\sqrt{mn}}{2}. \label{in1}%
\end{equation}
Equality holds in (\ref{in1}) if and only if the matrix $2A-J_{m,n}$ is a
regular $\left(  -1,1\right)  $-matrix, with row sums equal to $n/\sqrt{k}$,
and with exactly $k$ nonzero singular values, which are equal to $\sqrt{mn/k}$.
\end{theorem}

\begin{proof}
Let $H:=2A-J_{m,n}$ and note that $\left\Vert H\right\Vert _{\max}\leq1$.
First, the AM-QM inequality implies that
\begin{align*}
\left\Vert H\right\Vert _{\left[  k\right]  }  &  =\sigma_{1}\left(  H\right)
+\cdots+\sigma_{k}\left(  H\right)  \leq\sqrt{k\left(  \sigma_{1}^{2}\left(
H\right)  +\cdots+\sigma_{k}^{2}\left(  H\right)  \right)  }\\
&  \leq\sqrt{k\left(  \sigma_{1}^{2}\left(  H\right)  +\cdots+\sigma_{n}%
^{2}\left(  H\right)  \right)  }=\sqrt{k}\left\Vert H\right\Vert _{2}\\
&  \leq\sqrt{kmn}.
\end{align*}
Hence, the triangle inequality implies that%
\[
\left\Vert 2A\right\Vert _{\left[  k\right]  }=\left\Vert H+J_{m,n}\right\Vert
_{\left[  k\right]  }\leq\left\Vert H\right\Vert _{\left[  k\right]
}+\left\Vert J_{m,n}\right\Vert _{\left[  k\right]  }\leq\sqrt{kmn}+\sqrt
{mn},
\]
proving (\ref{in1}).

Now, suppose that a matrix $A$ satisfies the premises and forces equality in
(\ref{in1}), and let $H:=2A-J_{m,n}.$ Obviously $\left\Vert H\right\Vert
_{2}=\sqrt{mn},$ and so $H$ is a $\left(  -1,1\right)  $-matrix with
\[
\sigma_{1}\left(  H\right)  =\cdots=\sigma_{k}\left(  H\right)  =\sqrt
{mn/k}\text{ \ \ and \ \ }\sigma_{i}\left(  H\right)  =0\text{ \ \ for
\ }i>k.
\]
It remains to show that $H$ is regular and calculate its row sums. Indeed, for
every $i=2,\ldots,n,$ Weyl's inequality (\ref{Wsin}) implies that
\[
\sigma_{i}\left(  2A\right)  \leq\sigma_{i-1}\left(  H\right)  +\sigma
_{2}\left(  J_{m,n}\right)  =\sqrt{mn/k}.
\]
Since equality holds in (\ref{in1}), it follows that
\begin{align*}
\sigma_{1}\left(  2A\right)   &  \geq\sqrt{kmn}+\sqrt{mn}-\left(  k-1\right)
\sqrt{mn/k}=\sqrt{mn/k}+\sqrt{mn}\\
&  =\sigma_{1}\left(  H\right)  +\sigma_{1}\left(  J_{m,n}\right)  .
\end{align*}
On the other hand, letting $\mathbf{x}\in\mathbb{R}^{n}$ and $\mathbf{y}%
\in\mathbb{R}^{m}$ be unit singular vectors to $\sigma_{1}\left(  2A\right)
,$ we see that
\begin{equation}
\sigma_{1}\left(  2A\right)  =\left\langle 2A\mathbf{x},\mathbf{y}%
\right\rangle =\left\langle H\mathbf{x},\mathbf{y}\right\rangle +\left\langle
J_{n}\mathbf{x},\mathbf{y}\right\rangle \leq\sigma_{1}\left(  H\right)
+\sigma_{1}\left(  J_{n}\right)  . \label{in4}%
\end{equation}
Therefore, equality holds in (\ref{in4}), and so $\mathbf{x}$ and $\mathbf{y}$
are singular vectors to $\sigma_{1}\left(  H\right)  $ and to $\sigma
_{1}\left(  J_{m,n}\right)  $ as well; hence, $\mathbf{x}=n^{-1/2}%
\mathbf{j}_{n}$ and $\mathbf{y}=m^{-1/2}\mathbf{j}_{m}.$ Thus, $H$ is regular
and its row sums are equal to
\[
\frac{1}{m}\left\langle H\mathbf{j}_{n},\mathbf{j}_{m}\right\rangle =\frac
{1}{m}\sqrt{nm}\left\langle H\mathbf{x},\mathbf{y}\right\rangle =\sqrt
{\frac{n}{m}}\sigma_{1}\left(  H\right)  =\sqrt{\frac{n}{m}}\sqrt
{mn/k}=n/\sqrt{k},
\]
as claimed.

We omit the proof that the given conditions on $A$ force equality in
(\ref{in1}).
\end{proof}

The class $\mathbb{R}_{m,n}\left(  k\right)  $ of regular $\left(
-1,1\right)  $-matrices, with row sums equal to $n/\sqrt{k}$ and with $k$
nonzero and equal singular values, seems quite remarkable. Note that
$\mathbb{R}_{m,n}\left(  k\right)  $ consists of matrices of rank $k,$ and
some of them are Kronecker products of partial Hadamard matrices. Yet, in many
cases $\mathbb{R}_{m,n}\left(  k\right)  $ contains matrices that are not
Kronecker products of partial Hadamard matrices. The complete characterization
of $\mathbb{R}_{m,n}\left(  k\right)  $ is an open problem.\medskip

For graphs we get a straightforward corollary of Theorem \ref{tNik}:

\begin{theorem}
\label{thKF}If $n\geq k\geq2,$ and $G$ is a graph of order $n$ with adjacency
matrix $A,$ then
\begin{equation}
\left\Vert G\right\Vert _{\left[  k\right]  }\leq\frac{n\sqrt{k}}{2}+\frac
{n}{2}. \label{Nikbo}%
\end{equation}
Equality holds in (\ref{Nikbo}) if and only if $2A-J_{n}$ is a regular,
symmetric $\left(  -1,1\right)  $-matrix, with $-1$ along the diagonal, with
row sums equal to $n/\sqrt{k}$, and with exactly $k$ nonzero singular values,
which are equal to $n/\sqrt{k}$.
\end{theorem}

Clearly, Theorem \ref{thKF} is a far reaching generalization of the Koolen and
Moulton upper bound (\ref{KM}). The condition for equality, however, is
different, and we can say more about it. To this end, we need some preliminary
work, which was recently reported in \cite{Nik15b}, and which we state in the
following section. For the missing proofs the reader is referred to
\cite{Nik15b}.

\subsection{\label{sec eH}An extension of symmetric Hadamard matrices}

To study maximal Ky Fan norms of graphs we need symmetric $\left(
-1,1\right)  $-matrices whose nonzero singular values are equal. Such matrices
have distinctive properties and deserve close attention, for they extend
symmetric Hadamard matrices in a nontrivial way.\medskip

Thus, write $n\left(  A\right)  $ for the order of a square matrix $A,$ and
let $\mathbb{S}_{k}$ be the set of symmetric $\left(  -1,1\right)  $-matrices
with $\sigma_{k}\left(  A\right)  =n\left(  A\right)  /\sqrt{k}.$

First, note that if an $n\times n$ matrix $A,$ with all entries of modulus
$1,$ satisfies $\sigma_{k}\left(  A\right)  =n/\sqrt{k},$ then
\[
\sigma_{1}\left(  A\right)  =\cdots=\sigma_{k}\left(  A\right)  =n/\sqrt
{k}\text{\ \ and\ \ }\ \sigma_{i}\left(  A\right)  =0\text{ \ \ for \ }k<i\leq
n.
\]

Further, if $A\in$ $\mathbb{S}_{k},$ and if a matrix $B$ can be obtained by
permutations or negations performed simultaneously on the rows and columns of
$A,$ then $B\in$ $\mathbb{S}_{k}$ as well; the reason is that $A$ and $B$ have
the same singular values.\medskip

For reader's sake, we list six basic properties of $\mathbb{S}_{k},$ whose
proofs are omitted:\medskip

(1) If $A\in\mathbb{S}_{k}$, then $-A\in\mathbb{S}_{k}.$

(2) If $A$ is a symmetric $\left(  -1,1\right)  $-matrix of rank 1, then
$A\in\mathbb{S}_{1}$; thus, $J_{n}\in\mathbb{S}_{1}.$

(3) If $H$ is a symmetric Hadamard matrix of order $k,$ then $H\in
\mathbb{S}_{k}$; thus $\mathbb{S}_{2}\neq\varnothing.$

(4) If $A\in\mathbb{S}_{k}$ and $B\in\mathbb{S}_{l},$ then $A\otimes
B\in\mathbb{S}_{kl}$; hence, if $\mathbb{S}_{k}\neq\varnothing,$ then
$\mathbb{S}_{2k}\neq\varnothing.$

(5) If $A\in\mathbb{S}_{k}$, then $A\otimes J_{n}$ $\in\mathbb{S}_{k}$ for any
$n\geq1$; hence, if $\mathbb{S}_{k}\neq\varnothing,$ then $\mathbb{S}_{k}$ is infinite.

(6) If $\mathbb{S}_{k}\neq\varnothing$, then $\mathbb{S}_{k}$ contains
matrices with all row sums equal to zero.\medskip

Using Properties (1)-(6), one can show that $\mathbb{S}_{k}$ contains matrices
with some special properties, as in the following proposition:

\begin{proposition}
\label{props}If $A\in\mathbb{S}_{k},$ then there is a $B\in\mathbb{S}_{2k}$
with $\mathrm{tr}$\textrm{ }$B=0,$ that is, $B$ has exactly $k$ positive and
exactly $k$ negative eigenvalues.
\end{proposition}

Indeed, set
\[
H_{2}=\left[
\begin{array}
[c]{rr}%
1 & 1\\
1 & -1
\end{array}
\right]  ,
\]
and let $B=K\otimes\left(  H_{2}\otimes A\right)  ;$ obviously, $B\in$
$\mathbb{S}_{2k}$ and \textrm{tr }$B=0$.\bigskip

Property (3) shows that the sets $\mathbb{S}_{k}$ extend the class of
symmetric Hadamard matrices. However, $\mathbb{S}_{k}$ may also contain
matrices that does not come from Hadamard matrices; e.g., we shall see that
$\mathbb{S}_{s^{2}}\neq\varnothing$ for any natural number $s.\medskip$

The crucial question about $\mathbb{S}_{k}$ is the following one:

\begin{problem}
\label{probS}For which $k$ is $\mathbb{S}_{k}$ nonempty?
\end{problem}

As stated in Corollary \ref{cor3} below, if $k$ is odd and is not an exact
square, then $\mathbb{S}_{k}$ is empty. On the positive side, if $k$ is the
order of a symmetric Hadamard matrix, $\mathbb{S}_{k}$ is nonempty, so there
are infinitely many $k$ for which $\mathbb{S}_{k}$ is nonempty.

Here is a more definite assertion about $\mathbb{S}_{k}$:

\begin{proposition}
\label{proq}If $A\in\mathbb{S}_{k},$ then either $k$ is an exact square, or
$\mathrm{tr}$ $A=0$ and $A$ has the same number of positive and negative eigenvalues.
\end{proposition}

\begin{proof}
Let $A=\left[  a_{i,j}\right]  \in\mathbb{S}_{k}.$ Set $\lambda=\sigma
_{1}\left(  A\right)  ,$ $n=n\left(  A\right)  ,$ and let $\lambda_{1}%
\geq\cdots\geq\lambda_{n}$ be the eigenvalues of $A$. We have
\[
k\lambda^{2}=\sum_{\lambda_{i}>0}\lambda_{i}^{2}+\sum_{\lambda_{i}<0}%
\lambda_{i}^{2}=\mathrm{tr}\text{ }A^{2}=\sum_{i=1}^{n}\sum_{l=1}^{n}%
a_{i,j}^{2}=n^{2}.
\]
On the other hand, writing $n_{+}$ and $n_{-}$ for the number of positive and
negative eigenvalues of $A,$ we see that%
\[
\left(  n_{+}-n_{-}\right)  \lambda=\sum_{\lambda_{i}>0}\lambda_{i}%
+\sum_{\lambda_{i}<0}\lambda_{i}=\mathrm{tr}\text{ }A.
\]
If $\mathrm{tr}$ $A=0,$ then $n_{+}=n_{-},$ and so $A$ has the same number of
positive and negative eigenvalues. If $\mathrm{tr}$ $A\neq0,$ then $\lambda$
is rational, implying that $\sqrt{k}$ is rational, and so $k$ is an exact
square, completing the proof.
\end{proof}

\begin{corollary}
\label{cor3}If $k$ is odd and $\mathbb{S}_{k}$ is nonempty, then $k$ is an
exact square.
\end{corollary}

Let us note another corollary, which is a well-known fact about Hadamard
matrices corresponding to graphs with maximal energy:

\begin{corollary}
If $n$ is the order of a real symmetric Hadamard matrix with constant
diagonal, then $n$ is an exact square.
\end{corollary}

It seems important to determine if\ $\mathbb{S}_{k}$ is empty for small $k$,
say for $k\leq10.$ First, Corollary \ref{cor3} implies that $\mathbb{S}_{k}$
is empty for $k=3,5,$ and $7.$ As shown in Theorems \ref{thj} and \ref{thj1}
below, if $k$ is an exact square, then $S_{k}$ is nonempty. Thus, in view of
Properties (1)-(5), we see that $\mathbb{S}_{k}$ is nonempty for $k=1,2,4,8,$
and $9.$ The first unknown cases are $k=6$ and $k=10.$

\begin{question}
Is $\mathbb{S}_{6}$ empty?
\end{question}

Now, in Theorems \ref{thj} and \ref{thj1} we state two constructions of
symmetric $\left(  -1,1\right)  $-matrices showing that $\mathbb{S}_{k}$ is
nonempty if $k$ is an exact square. These constructions are proved in
\cite{Nik15b} by the Kharaghani method for Hadamard matrices \cite{Kha85},
which apparently is applicable to more general $\left(  -1,1\right)
$-matrices.\medskip

\begin{theorem}
\label{thj}For any integer $s\geq2$, there are an integer $n\geq s$ and a
symmetric $\left(  -1,1\right)  $-matrix $B$ of order $ns$ such that:

(i) $B$ has exactly $s^{2}$ nonzero eigenvalues, of which $\binom{s+1}{2}$ are
equal to $n,$ and $\binom{s}{2}$ are equal to $-n$;

(ii) the vector $\mathbf{j}_{ns}$ is an eigenvector of $B$ to the eigenvalue
$-n$;

(iii) the diagonal entries of $B$ are equal to $1.$
\end{theorem}

\begin{theorem}
\label{thj1}For any integer $s\geq2$, there are an integer $n\geq s$ and a
symmetric $\left(  -1,1\right)  $-matrix $B$ of order $ns$ such that:

(i) $B$ has exactly $s^{2}$ nonzero eigenvalues, of which $\binom{s+1}{2}-1$
are equal to $n,$ and $\binom{s}{2}+1$ are equal to $-n$;

(ii) all row sums of $B$ are equal to $-n$.
\end{theorem}

Note that Theorem \ref{thj1} and Proposition \ref{props} imply that for any
natural number $s$, the classes $\mathbb{S}_{s^{2}}$ and $\mathbb{S}_{2s^{2}}$
are nonempty. To summarize, let us state some explicit solutions to Problem
\ref{probS}:

\begin{proposition}
\label{propE}Let $s$ be a natural number and let $p$ be a prime power, with
$p=1$ $\operatorname{mod}$ $4.$ Then the classes $\mathbb{S}_{s^{2}},$
$\mathbb{S}_{2s^{2}},$ $\mathbb{S}_{2s^{2}\left(  p+1\right)  }$ and
$\mathbb{S}_{4s^{2}\left(  p+1\right)  }$ are nonempty.
\end{proposition}

\subsection{\label{sec 2}Maximal Ky Fan norms of graphs}

Armed with the results of Section \ref{sec eH}, we continue the study of
$\xi_{k}\left(  n\right)  $. It turns out that if bound (\ref{Nikbo}) is
attained, then $k$ is an exact square:

\begin{theorem}
\label{thk}If $n\geq k\geq2$ and $G$ is a graph of order $n$ such that
\[
\left\Vert G\right\Vert _{\left[  k\right]  }=\frac{n\sqrt{k}}{2}+\frac{n}%
{2},
\]
then $k$ is an exact square.
\end{theorem}

\begin{proof}
Suppose that $G$ is a graph of order $n$ with adjacency matrix $A$. For
reader's sake, we breeze through part of the proof of Theorem \ref{tNik}.

The matrix $H:=2A-J_{n}$ is a symmetric $\left(  -1,1\right)  $-matrix; hence
the AM-QM inequality implies that%
\begin{align*}
\sigma_{1}\left(  H\right)  +\cdots+\sigma_{k}\left(  H\right)   &  \leq
\sqrt{k\left(  \sigma_{1}^{2}\left(  H\right)  +\cdots+\sigma_{k}^{2}\left(
H\right)  \right)  }\leq\sqrt{k\left(  \sigma_{1}^{2}\left(  H\right)
+\cdots+\sigma_{n}^{2}\left(  H\right)  \right)  }\\
&  =\left\Vert H\right\Vert _{2}\sqrt{n}=n\sqrt{k}.
\end{align*}
Therefore, the triangle inequality for the Ky Fan $k$-norm implies that
\[
n\sqrt{k}+n=\left\Vert 2A\right\Vert _{\left[  k\right]  }\leq\left\Vert
H\right\Vert _{\left[  k\right]  }+\left\Vert J_{n}\right\Vert _{\left[
k\right]  }\leq n\sqrt{k}+n.
\]
Thus, equalities hold throughout the above line, and so, $H$ has $k$ nonzero
singular values, which are equal. We get
\[
\sigma_{k}\left(  H\right)  =n/\sqrt{k},
\]
implying that $H\in\mathbb{S}_{k}.$ On the other hand, $\mathrm{tr\,}%
H=n\neq0,$ so $H$ cannot have the same number of positive and negative
eigenvalues, and Proposition \ref{proq} implies that $k$ is an exact square.
\end{proof}

The matrices constructed in Theorem \ref{thj1} imply that the converse of
Theorem \ref{thk} is true whenever $k$ is an even square.

\begin{theorem}
\label{thck}Let $s\geq2$ be an even positive integer. There exists a positive
integer $p,$ such that for any positive integer $t,$ there is a graph $G$ of
order $n=spt$ with%
\[
\left\Vert G\right\Vert _{\left[  s^{2}\right]  }=\frac{sn}{2}+\frac{n}{2}.
\]

\end{theorem}

The proof of Theorem \ref{thck} can be found in \cite{Nik15b}. This theorem is
as good as one can get, but it is proved only for even $s$. However, if $s$ is
odd, $\xi_{s^{2}}\left(  n\right)  $ is just slightly below the upper bound
(\ref{Nikbo}), as implied by the following more general theorem:

\begin{theorem}
\label{thck1}Suppose that $\mathbb{S}_{k}$ contains a regular matrix $B$ with
nonzero row sums, say of order $k$. Then for any positive integer $t,$ there
is a graph $G$ of order $n=kt$ such that
\begin{equation}
\left\Vert G\right\Vert _{\left[  k\right]  }\geq\frac{n\sqrt{k}}{2}+\frac
{n}{2}-k. \label{in5}%
\end{equation}

\end{theorem}

The proof of Theorem \ref{thck1} can be found in \cite{Nik15b}.

Note that Theorem \ref{thj1} implies that for any integer $s\geq2,$ the set
$\mathbb{S}_{s^{2}}$ contains a regular matrix with nonzero row sums.
Therefore, the premise of Theorem \ref{thck1} holds for every $k=s^{2},$ where
$s$ is an integer and $s\geq2.$\medskip

Dividing both sides of (\ref{in5}) by $n$ and letting $n\rightarrow\infty,$ we
obtain the following corollary:

\begin{corollary}
If \ $\mathbb{S}_{k}$ contains a regular matrix $B$ with nonzero row sums,
then
\[
\xi_{k}=\frac{\sqrt{k}}{2}+\frac{1}{2}.
\]

\end{corollary}

Therefore, $\xi_{4}=3/2,$ $\xi_{9}=2,$ $\xi_{16}=5/2,$ $\xi_{25}=3,$
etc.\medskip

Since Theorem \ref{thk} does not shed any light on the case of non-square $k$,
we state the following straightforward conjecture:

\begin{conjecture}
There exist infinitely many positive integers $k$ such that
\[
\xi_{k}<\frac{\sqrt{k}}{2}+\frac{1}{2}.
\]

\end{conjecture}

We end up this section with the easy asymptotics%
\[
\frac{\sqrt{k}}{2}\leq\xi_{k}\leq\frac{\sqrt{k}}{2}+\frac{1}{2}.
\]

\subsection{\label{FD}Nordhaus-Gaddum problems for Ky Fan norms}

In this section we extend some results from Section \ref{NGtr} to Ky Fan norms
of rectangular nonnegative matrices. Most of these results were proved in
\cite{NiYu13}, but some new developments are presented here for the first
time, so their proofs are given in full.

\begin{theorem}
\label{NG3}If $n\geq m\geq k\geq2$ and $A$ is an $m\times n$ nonnegative
matrix with $\left\Vert A\right\Vert _{\max}\leq1,$ then
\begin{equation}
\left\Vert A\right\Vert _{\left[  k\right]  }+\left\Vert J_{m,n}-A\right\Vert
_{\left[  k\right]  }\leq\sqrt{\left(  k-1\right)  mn}+\sqrt{mn}.
\label{th3in}%
\end{equation}
Equality holds if only if $A$ is a regular $\left(  0,1\right)  $-matrix such
that%
\[
\sigma_{1}(A)=\sigma_{1}(\overline{A})=\frac{\sqrt{mn}}{2}%
\]
and
\begin{equation}
\sigma_{2}(A)=\cdots=\sigma_{k}(A)=\sigma_{2}(\overline{A})=\cdots=\sigma
_{k}(\overline{A})=\frac{1}{2}\sqrt{\frac{mn}{k-1}}. \label{eqs}%
\end{equation}

\end{theorem}

\begin{proof}
Set for short $\overline{A}=J_{m,n}-A.$ Following a familiar path, we see
that
\begin{align}
&  \sigma_{1}\left(  A\right)  +\sigma_{1}(\overline{A})+%
{\displaystyle\sum\limits_{i=2}^{k}}
\left(  \sigma_{i}\left(  A\right)  +\sigma_{i}(\overline{A})\right)
\nonumber\\
&  \leq\sigma_{1}\left(  A\right)  +\sigma_{1}(\overline{A})+\sqrt{2\left(
k-1\right)
{\displaystyle\sum\limits_{i=2}^{k}}
\left(  \sigma_{i}^{2}\left(  A\right)  +\sigma_{i}^{2}(\overline{A})\right)
}\label{aq1}\\
&  \leq\sigma_{1}\left(  A\right)  +\sigma_{1}(\overline{A})+\sqrt{2\left(
k-1\right)  \left(  \left\Vert A\right\Vert _{2}^{2}+\left\Vert \overline
{A}\right\Vert _{2}^{2}-\sigma_{1}^{2}\left(  A\right)  -\sigma_{1}%
^{2}(\overline{A})\right)  }\nonumber\\
&  \leq\sigma_{1}\left(  A\right)  +\sigma_{1}(\overline{A})+\sqrt{2\left(
k-1\right)  \left(  mn-\frac{(\sigma_{1}\left(  A\right)  +\sigma
_{1}(\overline{A}))^{2}}{2}\right)  .}\nonumber
\end{align}
Since the function
\[
f(x)=x+\sqrt{2(k-1)\left(  mn-x^{2}/2\right)  }%
\]
is decreasing in $x$ for $x\geq\sqrt{mn}$, and also
\begin{equation}
\sigma_{1}(A)+\sigma_{1}(\overline{A})\geq\frac{1}{\sqrt{mn}}\left\langle
A\mathbf{j}_{n},\mathbf{j}_{m}\right\rangle +\frac{1}{\sqrt{mn}}\left\langle
\overline{A}\mathbf{j}_{n},\mathbf{j}_{m}\right\rangle =\frac{1}{\sqrt{mn}%
}\left\langle J_{m,n}\mathbf{j}_{n},\mathbf{j}_{m}\right\rangle =\sqrt{mn},
\label{1}%
\end{equation}
we see that
\[
\left\Vert A\right\Vert _{\ast}+\left\Vert \overline{A}\right\Vert _{\ast}%
\leq\sqrt{\left(  k-1\right)  mn}+\sqrt{mn},
\]
completing the proof of (\ref{th3in}).

If equality holds in (\ref{th3in}), then $\left\Vert A\right\Vert _{2}%
^{2}+\left\Vert \overline{A}\right\Vert _{2}^{2}=mn,$ which means that $A$ is
a $\left(  0,1\right)  $-matrix, and so is $\overline{A}$. Further,%
\[
\sigma_{1}^{2}\left(  A\right)  +\sigma_{1}^{2}(\overline{A})=(\sigma
_{1}\left(  A\right)  +\sigma_{1}(\overline{A}))^{2}/2=\frac{mn}{2},
\]
which means that $\sigma_{1}(A)=\sigma_{1}(\overline{A})=\sqrt{mn}/2.$ Hence
equalities hold throughout (\ref{1}), and so $A$ is regular.

Finally, the AM-QM inequality applied in (\ref{aq1}) is equality precisely if
\[
\sigma_{2}(A)=\cdots=\sigma_{k}(A)=\sigma_{2}(\overline{A})=\cdots=\sigma
_{k}(\overline{A}),
\]
and since
\[
\sigma_{2}^{2}(A)+\cdots+\sigma_{k}^{2}(A)+\sigma_{2}^{2}(\overline{A}%
)+\cdots+\sigma_{k}^{2}(\overline{A})=\frac{mn}{2},
\]
we obtain (\ref{eqs}).

We omit the proof that the given conditions on $A$ force equality in
(\ref{th3in}). Theorem \ref{NG3} is proved.
\end{proof}

It seems hard to give a constructive description of all matrices $A$ forcing
equality in (\ref{th3in}), so we raise the following problem:

\begin{problem}
Let $n\geq m\geq k\geq2.$ Find a constructive description of all $m\times n$
nonnegative matrices $A$ with $\left\Vert A\right\Vert _{\max}\leq1$ such
that
\[
\left\Vert A\right\Vert _{\left[  k\right]  }+\left\Vert J_{m,n}-A\right\Vert
_{\left[  k\right]  }=\sqrt{\left(  k-1\right)  mn}+\sqrt{mn}.
\]

\end{problem}

Nevertheless, here is a construction showing that (\ref{th3in}) is exact in a
rich set of cases.

\begin{theorem}
\label{th3a}Let $t\geq k-1\geq2,$ and let $B$ be a $\left(  k-1\right)  \times
t$ partial Hadamard matrix. Let $p,q\geq1$ be arbitrary integers and set
$m=2\left(  k-1\right)  p$ and $n=2tq.$ Then, there exists a $\left(
0,1\right)  $-matrix $A$ of size $m\times n$ such that
\[
\left\Vert A\right\Vert _{\left[  k\right]  }+\left\Vert J_{m,n}-A\right\Vert
_{\left[  k\right]  }=\sqrt{\left(  k-1\right)  mn}+\sqrt{mn}.
\]

\end{theorem}

\begin{proof}
Set
\[
H=\left[
\begin{array}
[c]{rr}%
B & -B\\
-B & B
\end{array}
\right]
\]
and let
\[
A=\frac{1}{2}\left(  \left(  H\otimes J_{p,q}\right)  +J_{m,n}\right)  ,
\]
Obviously, $A$ is a $\left(  0,1\right)  $-matrix of size $m\times n.$ Our
goal is to show that $\sigma_{1}\left(  A\right)  =\sqrt{mn}/2$ and that%
\[
\sigma_{2}\left(  A\right)  =\cdots=\sigma_{k}\left(  A\right)  =\frac{1}%
{2}\sqrt{\frac{mn}{k-1}}.
\]
Recall that $B$ has $k-1$ nonzero singular values and they are equal to
$t/\sqrt{k-1}.$ Hence, $H$ has $k-1$ nonzero singular values, which are equal
to $2t/\sqrt{k-1},$ and so, $H\otimes J_{p,q}$ has $k-1$ nonzero singular
values, which are equal to
\[
2t\sqrt{\frac{pq}{k-1}}=\sqrt{\frac{2tp2tq}{k-1}}=\sqrt{\frac{mn}{k-1}}.
\]
Clearly, the row and column sums of $H$ are $0,$ and thus $0$ is a singular
value of $H$ with singular vectors $\left(  2\left(  k-1\right)  \right)
^{-1/2}\mathbf{j}_{2\left(  k-1\right)  }$ and $\left(  2t\right)
^{-1/2}\mathbf{j}_{2t}$; hence, $0$ is a singular value of $H\otimes J_{p,q}$
with singular vectors $m^{-1/2}\mathbf{j}_{m}$ and $n^{-1/2}\mathbf{j}_{n}.$
Since the unique nonzero singular value of $J_{m,n}$ is $\sqrt{mn},$ with
singular vectors $m^{-1/2}\mathbf{j}_{m}$ and $n^{-1/2}\mathbf{j}_{n},$ it is
obvious that
\[
\sigma_{1}\left(  A\right)  =\frac{\sqrt{mn}}{2}\text{ \ and \ }\sigma
_{2}\left(  A\right)  =\cdots=\sigma_{k}\left(  A\right)  =\frac{1}{2}%
\sqrt{\frac{mn}{k-1}}.
\]
Hence,
\[
\left\Vert A\right\Vert _{\left[  k\right]  }=\frac{\sqrt{mn}}{2}+\frac{1}%
{2}\sqrt{\frac{mn}{k-1}}\left(  k-1\right)  =\frac{\sqrt{\left(  k-1\right)
mn}}{2}+\frac{\sqrt{mn}}{2}.
\]
On the other hand,%
\[
J_{m,n}-A=\frac{1}{2}\left(  \left(  -H\otimes J_{p,q}\right)  +J_{m,n}%
\right)  ,
\]
and so $J_{m,n}-A$ has the same singular values as $A,$ for $-B\in
\mathbb{S}_{k-1}$ as well. This completes the proof of Theorem \ref{th3a}.
\end{proof}

Note that Theorems \ref{NG3} and \ref{th3a} are matrix results, easier than
the corresponding results for graphs. To get a meaningful statement for
graphs, we propose a matrix problem, which probably can be solved following
the proof of Theorem \ref{NG1}:

\begin{problem}
\label{ps}Let $n\geq k\geq2$ and let $A$ be a $n\times n$ symmetric
nonnegative matrix with $\left\Vert A\right\Vert _{\max}\leq1$ and with zero
diagonal. Find the maximum of
\[
\left\Vert A\right\Vert _{\left[  k\right]  }+\left\Vert J_{m,n}%
-I_{n}-A\right\Vert _{\left[  k\right]  }.
\]

\end{problem}

Finally, note that Theorem \ref{NG3} is stated and proved for $k\geq2.$ As it
turns out, the Ky Fan $1$ norm (i.e., the operator norm) is completely
different (for a proof see \cite{NiYu13}).

\begin{theorem}
\label{NG4}If $A$ is an $m\times n$ nonnegative matrix with $\left\Vert
A\right\Vert _{\max}\leq1$, then,
\begin{equation}
\sigma_{1}\left(  A\right)  +\sigma_{1}\left(  J_{m,n}-A\right)  \leq
\sqrt{2mn}, \label{sig1in}%
\end{equation}
with equality holding if and only if $mn$ is even, and $A$ is a $\left(
0,1\right)  $-matrix with precisely $mn/2$ ones that are contained either in
$n/2$ columns or in $m/2$ rows.
\end{theorem}

\subsection{\label{KFch}Ky Fan norms and some graph parameters}

This section contains a few relations of Ky Fan norms with basic graph
parameters, which lead to some challenging open problems.

Recall the well-known result of Caporossi, Cvetkovi\'{c}, Gutman, and Hansen
\cite{CCGH99}:\medskip

\emph{If }$G$\emph{ is a graph, then }%
\begin{equation}
\left\Vert G\right\Vert _{\ast}\geq2\lambda_{1}\left(  G\right)  , \label{Cap}%
\end{equation}
\emph{with equality holding if and only if }$G$\emph{ is a complete
multipartite graph with possibly some isolated vertices.}\medskip

We shall uncover the role of the chromatic number in (\ref{Cap}): Let $G$ be a
graph of order $n$ and chromatic number $\chi.$ Recall the famous inequality
of Hoffman \cite{Hof70}
\begin{equation}
\lambda_{1}\left(  G\right)  \leq\left\vert \lambda_{n}\left(  G\right)
\right\vert +\cdots+\left\vert \lambda_{n-\chi+2}\left(  G\right)  \right\vert
,\label{Hin}%
\end{equation}
which obviously implies that
\[
\lambda_{1}\left(  G\right)  \leq\sigma_{2}\left(  G\right)  +\cdots
+\sigma_{\chi}\left(  G\right)  .
\]
Thus, Hoffman's inequality (\ref{Hin}) strengthens (\ref{Cap}) as follows:

\begin{theorem}
\label{tHof}If $G$ is a graph with chromatic number $\chi\geq2,$ then%
\begin{equation}
\left\Vert G\right\Vert _{\left[  \chi\right]  }\geq2\lambda_{1}\left(
G\right)  . \label{HKF}%
\end{equation}

\end{theorem}

Note that if $G$ is a complete $\chi$-partite graph with possibly some
isolated vertices, then equality holds in (\ref{HKF}). However, there are many
other cases of equality some of which are rather complicated. Clearly, if
equality holds in (\ref{HKF}), then equality holds in (\ref{Hin}), but the
converse is not obvious. Thus, we raise the following problem:

\begin{problem}
Which graphs $G$ satisfy the equality $\left\Vert G\right\Vert _{\left[
\chi\right]  }=2\lambda_{1}\left(  G\right)  $?
\end{problem}

In contrast to Theorem \ref{tHof}, for bipartite graphs we have
\[
\left\Vert G\right\Vert _{\left[  2\right]  }=2\lambda_{1}\left(  G\right)
\leq2\sqrt{\left\lfloor n^{2}/4\right\rfloor }.
\]
A similar inequality for $r$-partite graphs seems unknown, so we raise the
following problem:

\begin{problem}
What is the maximum of $\left\Vert G\right\Vert _{\left[  \chi\right]  }$ of
an $\chi$-chromatic graph of order $n.$
\end{problem}

Next, recall that in \cite{GHK01}, Gregory, Hershkowitz and Kirkland proved
the following theorem:

\begin{theorem}
\label{tTfree}If $G$ is a graph with $m$ edges and largest eigenvalue
$\lambda$, then%
\begin{equation}
\left\Vert G\right\Vert _{\left[  2\right]  }\leq\lambda+\sqrt{2m-\lambda^{2}%
}\leq2\sqrt{m}. \label{KF2u}%
\end{equation}
Equality holds in (\ref{KF2u}) if and only if $G$ is a complete bipartite
graph with possibly some isolated vertices.
\end{theorem}

As the authors of \cite{GHK01} note: if $m>\left\lfloor n^{2}/4\right\rfloor
,$ bound (\ref{KF2u}) is never attained. Thus, let us reiterate one of the
problems in \cite{GHK01}:

\begin{problem}
If $G$ is a graph of order $n,$ with $m$ edges, how large can $\left\Vert
G\right\Vert _{\left[  2\right]  }$ be?
\end{problem}

Further, note that Mantel's theorem \cite{Man07} gives the following immediate
corollary of inequality (\ref{KF2u}).

\begin{corollary}
If $G$ is a triangle-free graph of order $n,$ then%
\[
\left\Vert G\right\Vert _{\left[  2\right]  }<2\sqrt{\left\lfloor
n^{2}/4\right\rfloor }.
\]
unless $G=K_{\left\lfloor n/2\right\rfloor ,\left\lceil n/2\right\rceil }$.
\end{corollary}

The extension of this statement to $K_{r+1}$-free graphs (graphs without a
complete subgraph of order $r+1$) for $r\geq3$ is nowhere in sight. We state
two versions of such a problem:

\begin{problem}
If $r\geq3$ and $G$ is a $K_{r+1}$-free graph with $m$ edges, how large can
$\left\Vert G\right\Vert _{\left[  r\right]  }$ be?
\end{problem}

\begin{problem}
If $r\geq3$ and $G$ is a $K_{r+1}$-free graph of order $n,$ how large can
$\left\Vert G\right\Vert _{\left[  r\right]  }$ be?
\end{problem}

Finally, it is interesting to study the following problems, which are typical
for the study of graph energy.

\begin{problem}
Let $k\geq2$ and $G$ let be a connected graph of sufficiently large order $n.$
Is it true that
\[
\left\Vert G\right\Vert _{\left[  k\right]  }\geq\left\Vert P_{n}\right\Vert
_{\left[  k\right]  },
\]
where $P_{n}$ is the path of order $n?$
\end{problem}

\begin{problem}
Let $k\geq2$ and $T$ be a tree of sufficiently large order $n.$ How large
$\left\Vert T\right\Vert _{\left[  k\right]  }$ can be?
\end{problem}

\section{\label{secS}The Schatten norms}

This section is dedicated to the Schatten norms of graphs and matrices. Unlike
the somewhat choppy Ky Fan $k$-norms, the Schatten $p$-norms are rather
smooth; thus, many results on graph energy seamlessly extend to Schatten
norms.\medskip

Since the parameter $p$ in the Schatten $p$-norm may be any real number in the
interval $\left[  1,\infty\right)  ,$ some new problems arise, for which graph
energy provides no clues at all. Such problems are discussed in the opening
Section \ref{SSS}, with the rest of Section \ref{secS} dedicated to more
traditional topics.\medskip

In Section \ref{SSm}, we discuss extremal Schatten $p$-norms of matrices and
their relations to Hadamard matrices. We show that the use of $p$ allows to
obtain lower and upper bounds with the same argument; in particular, two bound
converters similar to Proposition \ref{p1} are obtained with the same
proof.\medskip

In Section \ref{SSg}, we discuss bounds on Schatten $p$-norms of graphs, most
of which come from matrix bounds, combined with some results about graphs.
Once more it is shown that, depending on $p,$ lower bounds become upper, and
vice versa.\medskip

Section \ref{shin} contributes to the popular topic of spectral moments in
graph energy. We shall show that H\"{o}lder's inequality and Schatten norms
provide a very convenient setup for this topic, and shall extend several known
results.\medskip

In Section \ref{SSr}, we give a straightforward generalization of the results
on the trace norms of $r$-partite graphs and matrices stated in Section
\ref{secMTN}, this time with proofs. As before, Kronecker products of
conference and Hadamard matrices provide the tightest known bounds.\medskip

In the brief Section \ref{Snt} we shall raise several question about extremal
Schatten norms of trees.\medskip

Finally, in Section \ref{sec1.1}, we discuss the Schatten $p$-norm of almost
all graphs, which turn to be as highly concentrated as the energy.

\subsection{\label{SSS}The Schatten $p$-norm as a function of $p$}

As mentioned above, Schatten $p$-norms open a new direction of research, with
no roots in graph energy. To elaborate this point, given a graph $G,$ define
the function $f_{G}\left(  x\right)  $ for any $x\geq1$ as%
\[
f_{G}\left(  x\right)  :=\left\Vert G\right\Vert _{x}.
\]
The energy of $G$ is just $f_{G}\left(  1\right)  ,$ but the function
$f_{G}\left(  x\right)  $ delivers much more. Thus, let us give some analytic
statements, with no analogs in the study of graph energy.

\begin{proposition}
\label{ps1}For any graph $G,$ the function $f_{G}\left(  x\right)  $ is
differentiable in $x$.
\end{proposition}

The proof of this proposition is simple calculus, so we omit it. Here is
another fact:

\begin{proposition}
\label{ps2}For any nonempty graph $G,$ the function $f_{G}\left(  x\right)  $
is decreasing in $x$.
\end{proposition}

\begin{proof}
Indeed, let $G$ be\ graph of order $n$ with singular values $\sigma_{1}%
,\ldots,\sigma_{n}.$ If $x<y,$ then%
\begin{align*}
f_{G}\left(  x\right)   &  =\sigma_{1}\left(  1+\left(  \frac{\sigma_{2}%
}{\sigma_{1}}\right)  ^{x}+\cdots+\left(  \frac{\sigma_{n}}{\sigma_{1}%
}\right)  ^{x}\right)  ^{1/x}\geq\sigma_{1}\left(  1+\left(  \frac{\sigma_{2}%
}{\sigma_{1}}\right)  ^{y}+\cdots+\left(  \frac{\sigma_{n}}{\sigma_{1}%
}\right)  ^{y}\right)  ^{1/x}\\
&  >\sigma_{1}\left(  1+\left(  \frac{\sigma_{2}}{\sigma_{1}}\right)
^{y}+\cdots+\left(  \frac{\sigma_{n}}{\sigma_{1}}\right)  ^{y}\right)
^{1/y}\\
&  =f_{G}\left(  y\right)  .
\end{align*}

\end{proof}

Using calculus, we obtain another property of $f_{G}\left(  x\right)  :$

\begin{proposition}
\label{ps3}For any graph $G,$ the limit $\lim_{x\rightarrow\infty}f_{G}\left(
x\right)  $ exists and is equal to $\lambda_{1}\left(  G\right)  .$
\end{proposition}

Propositions \ref{ps2} and \ref{ps3} imply that the range of $f_{G}\left(
x\right)  $ is the interval $\left(  \lambda_{1}\left(  G\right)  ,\left\Vert
G\right\Vert _{\ast}\right]  .$ \medskip

Next, we restate a basic and well-known fact in spectral graph theory:

\begin{proposition}
\label{ps4}If $G$ is a graph with $m$ edges, then $f_{G}\left(  2\right)
=\sqrt{2m}.$ Furthermore, for any $k>1,$ the number of closed walks of length
$2k$ of a graph $G$ is equal to $\left(  f_{G}\left(  2k\right)  \right)
^{2k}/4k.$
\end{proposition}

Given that the number of edges and so many other graph parameters can be read
from the function $f_{G}\left(  2\right)  ,$ it is natural to ask the
following question:

\begin{question}
\label{pr1}Which graph properties can be determined from the function
$f_{G}\left(  x\right)  $ alone?
\end{question}

Let us note that the order of a graph $G$ cannot be determined from
$f_{G}\left(  x\right)  $, because adding or removing isolated vertices does
not change $f_{G}\left(  x\right)  .$ Here is another example to the same
effect, for connected graphs: let $F:=K_{n,n}$ and $H:=K_{1,n^{2}}$;
obviously, $f_{F}\left(  x\right)  =f_{H}\left(  x\right)  =n2^{1/x},$ but
$v\left(  F\right)  =2n$ and $v\left(  H\right)  =n^{2}+1.\medskip$

Although we cannot infer the order of $G$ from $f_{G}\left(  x\right)  ,$ we
can find the singular spectrum of $G,$ that is, the nonzero singular values of
$G$ together with their multiplicities.

\begin{proposition}
\label{pro5}There is a procedure that calculates the nonzero singular values
and their multiplicities of any graph $G$ if the function $f_{G}\left(
x\right)  $ is given$.$
\end{proposition}

\begin{proof}
Clearly we can find $\sigma_{1}\left(  G\right)  ,$ for $\sigma_{1}\left(
G\right)  =\lambda_{1}\left(  G\right)  =\lim_{x\rightarrow\infty}f_{G}\left(
x\right)  .$ Now, the multiplicity $k_{1}$ of $\sigma_{1}\left(  G\right)  $
clearly is equal to%
\[
\lim_{x\rightarrow\infty}\frac{f_{G}^{x}\left(  x\right)  }{\sigma_{1}%
^{x}\left(  G\right)  }.
\]
Next, we see that
\[
\lim_{x\rightarrow\infty}\text{ }\left(  f_{G}^{x}\left(  x\right)
-k_{1}\sigma_{1}^{x}\left(  G\right)  \right)  ^{1/x}=\sigma_{2}\left(
G\right)  ,
\]
and if $\sigma_{2}\left(  G\right)  \neq0,$ we can determine the multiplicity
of $\sigma_{2}\left(  G\right)  $ as%
\[
\lim_{x\rightarrow\infty}\text{ }\frac{f_{G}^{x}\left(  x\right)  -k_{1}%
\sigma_{1}^{x}\left(  G\right)  }{\sigma_{2}^{x}\left(  G\right)  }.
\]
Iterating this argument, we obtain all nonzero singular values of $G,$ along
with their multiplicities.
\end{proof}

Put in a different way, $f_{G}\left(  x\right)  $ carries the same information
as the singular spectrum of $G.$ Obviously, cospectral graphs have the same
singular spectrum, but the converse may not be true. On the other hand, any
two graphs with the same singular spectrum are coenergetic, but the converse
may not be true. Hence, $f_{G}\left(  x\right)  $ introduces a new type of
equivalence among graphs, which needs further study.

Let us spell out the relevant definition and raise a problem.

\begin{definition}
Two graphs $G$ and $H$ are called \textbf{singularly cospectral} if they have
the same nonzero singular values with the same multiplicities.
\end{definition}

\begin{problem}
Find necessary and sufficient conditions two graphs to be singularly cospectral.
\end{problem}

Clearly Proposition \ref{pro5} implies that two graphs $G$ and $H$ are
singularly cospectral if and only if $f_{G}\left(  x\right)  =f_{H}\left(
x\right)  .$

\subsection{\label{SSm}Bounds on the Schatten $p$-norm of matrices}

We start with a general inequality for the Schatten norms of arbitrary
matrices in $M_{m,n}$.

\begin{theorem}
\label{MCgen}Let $n\geq m\geq2$ and $q>p\geq1.$ If $A\in M_{m,n},$ then
\begin{equation}
m^{-1/p}\left\Vert A\right\Vert _{p}\leq m^{-1/q}\left\Vert A\right\Vert _{q}.
\label{MCm}%
\end{equation}
If $A\neq0,$ equality holds in (\ref{MCm}) if and only if the following
equivalent conditions hold:

(i) $A$ has $m$ nonzero singular values which are equal;

(ii) $AA^{\ast}$ is a scalar multiple of the identity matrix $I_{m}$.
\end{theorem}

\begin{proof}
Inequality (\ref{MCm}) follows by applying the PM inequality to the singular
values of $A$; so it remains to prove the characterization of the matrices
that force equality in (\ref{MCm}). Let $A\in M_{m,n}.$

If $A$ satisfies either (i) or (ii), then obviously $A$ forces equality in
(\ref{MCm}).

Now, suppose that $A$ forces equality in (\ref{MCm}). Clearly, the condition
for equality in the PM inequality implies (i). To complete the proof, we shall
deduce (ii) from (i) using the idea of the proof of Proposition \ref{pro1}.

To begin with, note that the matrix $B:=AA^{\ast}$ is an $m\times m$ Hermitian
matrix, with $m$ equal eigenvalues, which are the squares of the singular
eigenvalues of $A$. Let $B=\left[  b_{i,j}\right]  $ and fix two distinct
$s,t\in\left[  m\right]  $. The $2\times2$ principal submatrix
\[
B^{\prime}=\left[
\begin{array}
[c]{cc}%
b_{s,s} & b_{s,t}\\
b_{t,s} & b_{t,t}%
\end{array}
\right]
\]
satisfies $b_{s,s}=\overline{b_{s,s}},$ $b_{t,t}=$ $\overline{b_{t,t}},$
$b_{s,t}=\overline{b_{t,s}},$ and so the eigenvalues of $B^{\prime}$ are
\begin{align*}
\lambda_{1}\left(  B^{\prime}\right)   &  =\frac{b_{s,s}+b_{t,t}+\sqrt{\left(
b_{s,s}-b_{t,t}\right)  ^{2}+4\left\vert b_{s,t}\right\vert ^{2}}}{2},\\
\lambda_{2}\left(  B^{\prime}\right)   &  =\frac{b_{s,s}+b_{t,t}-\sqrt{\left(
b_{s,s}-b_{t,t}\right)  ^{2}+4\left\vert b_{s,t}\right\vert ^{2}}}{2}.
\end{align*}
On the other hand, Cauchy's interlacing theorem implies that
\[
\lambda_{1}\left(  B\right)  \geq\lambda_{1}\left(  B^{\prime}\right)
\geq\lambda_{2}\left(  B^{\prime}\right)  \geq\lambda_{m}\left(  B\right)  ,
\]
which, in view of $\lambda_{1}\left(  B\right)  =\lambda_{m}\left(  B\right)
,$ implies that $b_{s,s}=$ $b_{t,t}$ and $b_{s,t}=b_{t,s}=0.$ Therefore, $B$
has a constant diagonal and all its off-diagonal entries are zero. \ Hence,
$B=aI_{m}$ for some $a>0,$ completing the proof of Theorem \ref{MCgen}.
\end{proof}

It is instructive to pair Theorem \ref{MCgen} with a similar lower bound on
$\left\Vert A\right\Vert _{p}$.

\begin{theorem}
\label{MCabsg}Let $n\geq m\geq2$ and $q>p\geq1.$ If $A\in M_{m,n},$ then
\begin{equation}
\left\Vert A\right\Vert _{p}\geq\left\Vert A\right\Vert _{q}. \label{MCmag}%
\end{equation}
Equality holds in (\ref{MCmag}) if and if $\sigma_{2}\left(  A\right)  =0.$
\end{theorem}

\begin{proof}
Let $A\in M_{m,n}$ and recall that
\begin{equation}
\sigma_{1}^{p}\left(  A\right)  +\cdots+\sigma_{m}^{p}\left(  A\right)
=\left\Vert A\right\Vert _{p}^{p}. \label{rest2}%
\end{equation}

To maximize $\left\Vert A\right\Vert _{p}$ subject to (\ref{rest2}), note that
$q/p>1,$ and so $x^{q/p}$ is a strictly convex function; hence, if
$x_{1},\ldots,x_{m}$ are nonnegative real numbers, with $x_{1}+\cdots
+x_{m}=s>0,$ then
\[
x_{1}^{q/p}+\cdots+x_{m}^{q/p}\leq s^{q/p},
\]
with equality if and only if all but one of the numbers $x_{1},\ldots,x_{m}$
are zero.

In our case, letting
\[
x_{1}=\sigma_{1}^{p}\left(  A\right)  ,\ldots,x_{m}=\sigma_{m}^{p}\left(
A\right)  \text{ \ \ and \ \ }s=\left\Vert A\right\Vert _{p}^{p},
\]
we see that
\begin{align*}
\left\Vert A\right\Vert _{q}^{q}  &  =\sigma_{1}^{q}\left(  A\right)
+\cdots+\sigma_{m}^{q}\left(  A\right)  =\left(  \sigma_{1}^{p}\left(
A\right)  \right)  ^{q/p}+\cdots+\left(  \sigma_{m}^{p}\left(  A\right)
\right)  ^{q/p}\\
&  \leq\left\Vert A\right\Vert _{p}^{q},
\end{align*}
with equality if and only if $\sigma_{1}\left(  A\right)  =\left\Vert
A\right\Vert _{p}$ and $\sigma_{2}\left(  A\right)  =\cdots=\sigma_{m}\left(
A\right)  =0.$
\end{proof}

Next we shall explore some upper bounds on $\left\Vert A\right\Vert _{p}.$ If
we restrict $\left\Vert A\right\Vert _{\max}$, Theorems \ref{MCgen} and
\ref{MCabsg} imply absolute bounds on $\left\Vert A\right\Vert _{p}.$ Note
that these bounds and their cases of equality differ significantly for $p<2$
and $p>2,$ so we state them in two separate theorems, whose proofs are
omitted. The reasons for the difference between the cases $p<2$ and $p>2$
elude us.

\begin{theorem}
\label{MCabs}Let $n\geq m\geq2$ and $2>p\geq1.$ If $A\in M_{m,n}$ and
$\left\Vert A\right\Vert _{\max}\leq1,$ then
\[
\left\Vert A\right\Vert _{p}\leq m^{1/p}n^{1/2}.
\]
Equality holds if and only if $A$ is a partial Hadamard matrix$.$
\end{theorem}

\begin{theorem}
\label{MCabs1}Let $n\geq m\geq2$ and $p>2.$ If $A\in M_{m,n}$ and $\left\Vert
A\right\Vert _{\max}\leq1,$ then
\[
\left\Vert A\right\Vert _{p}\leq\sqrt{mn}.
\]
Equality holds in if and if $\sigma_{2}\left(  A\right)  =0$, and all entries
of $A$ are of modulus $1.$
\end{theorem}

The triangle inequality for the Schatten norms implies neat bounds for
nonnegative matrices as well. Unfortunately, these bounds are tight, but not
sharp. Our first result is a generalization of (\ref{N1}):

\begin{theorem}
\label{KMSmn}Let $n\geq m\geq4,$ $2>p\geq1,$ and let $A$ be an $m\times n$
nonnegative matrix. If $\left\Vert A\right\Vert _{\max}\leq1,$ then
\begin{equation}
\left\Vert A\right\Vert _{p}\leq\frac{m^{1/p}n^{1/2}}{2}+\frac{\sqrt{mn}}{2},
\label{bo3}%
\end{equation}
If $2A-J_{m,n}$ is a regular partial Hadamard matrix, then
\[
\left\Vert A\right\Vert _{p}\geq\frac{m^{1/p}n^{1/2}}{2}.
\]

\end{theorem}

\begin{proof}
Suppose that $A$ satisfies the premises of the theorem and let $H:=2A-J_{m,n}%
$. Obviously $\left\Vert H\right\Vert _{\max}\leq1$ and the triangle
inequality implies that
\[
\left\Vert 2A\right\Vert _{p}=\left\Vert H+J_{n}\right\Vert _{p}\leq\left\Vert
H\right\Vert _{p}+\left\Vert J_{m,n}\right\Vert _{p}=\left\Vert H\right\Vert
_{p}+\sqrt{mn}.
\]
Theorem \ref{MCabs} gives $\left\Vert H\right\Vert _{p}\leq m^{1/p}n^{1/2}$
and (\ref{bo3}) follows.

If $H$ is a regular partial Hadamard matrix, then the singular values of $2A$
are $\sqrt{n}$ with multiplicity $m-1,$ and one singular value equal to
$\left\vert \sqrt{nm}\pm\sqrt{n}\right\vert .$ Since $m\geq4,$ it follows that
all $m$ nonzero singular values of $A$ are at least $\sqrt{n}$ and so%
\[
\left\Vert A\right\Vert _{p}\geq\frac{m^{1/p}n^{1/2}}{2}.
\]
completing the proof.
\end{proof}

The next statement complements Theorem \ref{KMSmn} for $p>2,$ and in fact
follows immediately from Theorem \ref{MCabs1}.

\begin{corollary}
Let $n\geq m\geq2,$ $p>2,$ and let $A$ be an $m\times n$ nonnegative matrix.
If $\left\Vert A\right\Vert _{\max}\leq1,$ then
\[
\left\Vert A\right\Vert _{p}\leq\sqrt{mn}.
\]
Equality holds if and if $A=J_{m,n}.$
\end{corollary}

Theorems \ref{MCgen}, \ref{MCabs}, and \ref{MCabs1} provide straightforward
bounds, good for their simplicity, but quite rigid. However, in exchange for a
somewhat complicated form, one can come up with more flexible inequalities, as
the following one:

\begin{proposition}
\label{p2}Let $n\geq m\geq2$ and $q>p\geq1.$ If $A\in M_{m,n},$ then
\begin{equation}
\left\Vert A\right\Vert _{p}^{p}\leq\sigma_{1}^{p}\left(  A\right)  +\left(
m-1\right)  ^{1-p/q}\left(  \left\Vert A\right\Vert _{q}^{q}-\sigma_{1}%
^{q}\left(  A\right)  \right)  ^{p/q}. \label{KMm}%
\end{equation}
Equality holds in (\ref{KMm}) if and only if $\sigma_{2}\left(  A\right)
=\cdots=\sigma_{m}\left(  A\right)  .$
\end{proposition}

\begin{proof}
The PM inequality, applied to the numbers $\sigma_{2}\left(  A\right)
,\ldots,\sigma_{m}\left(  A\right)  ,$ gives
\[
\left(  m-1\right)  ^{-1/p}\left(  \left\Vert A\right\Vert _{p}^{p}-\sigma
_{1}^{p}\left(  A\right)  \right)  ^{1/p}\leq\left(  m-1\right)
^{-1/q}\left(  \left\Vert A\right\Vert _{q}^{q}-\sigma_{1}^{q}\left(
A\right)  \right)  ^{1/q},
\]
and inequality (\ref{KMm}) follows, together with the condition for equality.
\end{proof}

Note that Proposition \ref{p2}, just like Proposition \ref{p1}, is a
\textquotedblleft bound generator\textquotedblright. We can see that by
repeating the argument after Proposition \ref{p1}. Therefore, Proposition
\ref{p2} can be used to convert lower bounds on $\sigma_{1}\left(  A\right)  $
into upper bounds on $\left\Vert A\right\Vert _{p}.$ In particular, in
\cite{Nik07a} an infinite family of lower bounds on $\sigma_{1}\left(
A\right)  $ is given; hence, with some restrictions we may get an unlimited
family of upper bounds on $\left\Vert A\right\Vert _{p}.$ Here is one of them,
which is based on the lower bound (\ref{lbs}).

\begin{proposition}
\label{p2s}Let $n\geq m\geq2,$ $q>p\geq1$ and $A\in M_{m,n}.$ If
\[
\left\vert \sum\nolimits_{i,j}a_{ij}\right\vert \geq m^{1/2-1/q}%
n^{1/2}\left\Vert A\right\Vert _{q},
\]
then
\[
\left\Vert A\right\Vert _{p}^{p}\leq\left(  mn\right)  ^{-p/2}\left\vert
\sum\nolimits_{i,j}a_{ij}\right\vert ^{p}+\left(  m-1\right)  ^{1-p/q}\left(
\left\Vert A\right\Vert _{q}^{q}-\left(  mn\right)  ^{-q/2}\left\vert
\sum\nolimits_{i,j}a_{ij}\right\vert ^{q}\right)  ^{p/q}.
\]
If equality holds in (\ref{b2}), then $\sigma_{2}\left(  A\right)
=\cdots=\sigma_{m}\left(  A\right)  ,$ the matrix $A$ is regular and
$\sigma_{1}\left(  A\right)  =\left(  mn\right)  ^{-1/2}\left\vert
\sum\nolimits_{i,j}a_{ij}\right\vert .$

If $A$ is nonnegative, then equality holds in (\ref{b2}) if and only if $A$ is
regular and $\sigma_{2}\left(  A\right)  =\cdots=\sigma_{m}\left(  A\right)  $.
\end{proposition}

We now turn to lower bounds on $\left\Vert A\right\Vert _{p}.$ Theorems
\ref{MCabsg}, and \ref{MCabs1} provide straightforward and simple bounds, but
in exchange for certain complication in form, we can obtain more flexible
inequalities, as in the following proposition, which is a twin of Proposition
\ref{p2}:

\begin{proposition}
\label{p3}Let $n\geq m\geq2$ and $p>q\geq1.$ If $A\in M_{m,n},$ then
\[
\left\Vert A\right\Vert _{p}^{p}\geq\sigma_{1}^{p}\left(  A\right)  +\left(
m-1\right)  ^{1-p/q}\left(  \left\Vert A\right\Vert _{q}^{q}-\sigma_{1}%
^{q}\left(  A\right)  \right)  ^{p/q}.
\]
Equality holds if and only if $\sigma_{2}\left(  A\right)  =\cdots=\sigma
_{m}\left(  A\right)  .$
\end{proposition}

The proof of Proposition \ref{p3} is absolutely the same as the proof of
Proposition \ref{p2}.

Note that Proposition \ref{p3}, is also a \textquotedblleft bound
generator\textquotedblright. The argument is the same as above. Hence,
Proposition \ref{p3} can be used to convert lower bounds on $\sigma_{1}\left(
A\right)  $ into lower bounds on $\left\Vert A\right\Vert _{p}.$ Here is a
lower bound on $\left\Vert A\right\Vert _{p}$ obtained from the lower bound
(\ref{lbs}).

\begin{proposition}
\label{p3s}Let $n\geq m\geq2,$ $p>q\geq1$ and $A\in M_{m,n}.$ If
\[
\left\vert \sum\nolimits_{i,j}a_{ij}\right\vert \geq m^{1/2-1/q}%
n^{1/2}\left\Vert A\right\Vert _{q},
\]
then
\[
\left\Vert A\right\Vert _{p}^{p}\geq\left(  mn\right)  ^{-p/2}\left\vert
\sum\nolimits_{i,j}a_{ij}\right\vert ^{p}+\left(  m-1\right)  ^{1-p/q}\left(
\left\Vert A\right\Vert _{q}^{q}-\left(  mn\right)  ^{-q/2}\left\vert
\sum\nolimits_{i,j}a_{ij}\right\vert ^{q}\right)  ^{p/q}.
\]
If equality holds in (\ref{b2}), then $\sigma_{2}\left(  A\right)
=\cdots=\sigma_{m}\left(  A\right)  ,$ the matrix $A$ is regular and
$\sigma_{1}\left(  A\right)  =\left(  mn\right)  ^{-1/2}\left\vert
\sum\nolimits_{i,j}a_{ij}\right\vert .$

If $A$ is nonnegative, then equality holds in (\ref{b2}) if and only if $A$ is
regular and $\sigma_{2}\left(  A\right)  =\cdots=\sigma_{m}\left(  A\right)  $.
\end{proposition}

It seems particularly interesting to characterize nonnegative matrices forcing
equality in Propositions \ref{p2s} and \ref{p3s}, because we obtain a far
reaching extension of design graphs.

\begin{problem}
\label{prob}Let $n\geq m\geq2.$ Give a constructive characterization of all
regular nonnegative $m\times n$ matrices $A$ with $\sigma_{2}\left(  A\right)
=\cdots=\sigma_{m}\left(  A\right)  .$
\end{problem}

Finally, let us note a natural extension of the simple bound (\ref{lobo}):

\begin{proposition}
If $2>p\geq1$ and $A$ is a matrix with rank greater than $1,$ then%
\begin{equation}
\left\Vert A\right\Vert _{p}^{p}\geq\sigma_{1}^{p}\left(  A\right)
+\frac{\left\Vert A\right\Vert _{2}^{2}-\sigma_{1}^{2}\left(  A\right)
}{\sigma_{2}^{2-p}\left(  A\right)  } \label{lom}%
\end{equation}
Equality holds if and only if all nonzero singular values of $A$ other than
$\sigma_{1}\left(  A\right)  $ are equal.
\end{proposition}

\begin{proof}
Let $\sigma_{1}\geq\sigma_{2}\geq\cdots\geq\sigma_{n}$ be the nonzero singular
values of $A.$ We see that
\begin{align*}
\left\Vert A\right\Vert _{p}^{p}  &  =\sigma_{1}^{p}+\sigma_{2}^{2}\sigma
_{2}^{p-2}+\cdots+\sigma_{n}^{2}\sigma_{n}^{p-2}\\
&  \geq\sigma_{1}^{p}+\frac{\sigma_{2}^{2}}{\sigma_{2}^{2-p}}+\cdots
+\frac{\sigma_{n}^{2}}{\sigma_{2}^{2-p}}=\sigma_{1}^{p}+\frac{\left\Vert
A\right\Vert _{2}^{2}-\sigma_{1}^{2}}{\sigma_{2}^{2-p}},
\end{align*}
and (\ref{lom}) follows.
\end{proof}

\subsection{\label{SSg}Bounds on the Schatten $p$-norm of graphs}

As mentioned above, usually results about graph energy extend to Schatten
norms rather smoothly. The goal of this subsection is to demonstrate this fact
on a few central results.

Let us start by restating Propositions \ref{p2} and \ref{p3} for graphs.

\begin{corollary}
\label{p2g}If $q>p\geq1$ and $G$ is a graph of order $n$ with largest
eigenvalue $\lambda,$ then
\[
\left\Vert G\right\Vert _{p}^{p}\leq\lambda^{p}+\left(  m-1\right)
^{1-p/q}\left(  \left\Vert G\right\Vert _{q}^{q}-\lambda^{q}\right)  ^{p/q}.
\]
Equality holds if and only if $\sigma_{2}\left(  G\right)  =\cdots=\sigma
_{n}\left(  G\right)  .$
\end{corollary}

\begin{corollary}
\label{p3g}If $p>q\geq1$ and $G$ is a graph of order $n$ with largest
eigenvalue $\lambda,$ then
\[
\left\Vert G\right\Vert _{p}^{p}\geq\lambda^{p}+\left(  m-1\right)
^{1-p/q}\left(  \left\Vert G\right\Vert _{q}^{q}-\lambda^{q}\right)  ^{p/q}.
\]
Equality holds if and only if $\sigma_{2}\left(  G\right)  =\cdots=\sigma
_{n}\left(  G\right)  .$
\end{corollary}

Two remarks are in place here: First, let us reiterate that\ the condition
$\sigma_{2}\left(  G\right)  =\cdots=\sigma_{n}\left(  G\right)  ,$ albeit
exact, is not constructive, and we are again lead to Problem \ref{prob}.
Second, just like Propositions \ref{p2} and \ref{p3}, Corollaries \ref{p2g}
and \ref{p3g} are bound converters, even better ones because it is easier to
find appropriate lower bounds for $\lambda,$ which will produce upper or lower
bounds on $\left\Vert G\right\Vert _{p}.$ To simplify the presentation we
focus exclusively on the pivotal case $q=2,$ thereby making use of the
equality $\left\Vert G\right\Vert _{2}^{2}=2e\left(  G\right)  $.

\subsubsection{Upper bounds for $2>p\geq1$}

For $q=2$ Corollary \ref{p2g} gives a practical and flexible bound:

\begin{corollary}
\label{pro4}Let $2>p\geq1.$ If $G$ is a graph of order $n$ and size $m,$ with
largest eigenvalue $\lambda,$ then
\begin{equation}
\left\Vert G\right\Vert _{p}^{p}\leq\lambda^{p}+\left(  n-1\right)
^{1-p/2}\left(  2m-\lambda^{2}\right)  ^{p/2}. \label{KMn1}%
\end{equation}
Equality holds in (\ref{KMn1}) if and only if $\sigma_{2}\left(  G\right)
=\cdots=\sigma_{n}\left(  G\right)  .$
\end{corollary}

Inequality (\ref{KMn1}) also is a \textquotedblleft bound
generator\textquotedblright, which can convert lower bounds on $\lambda$ into
upper bounds on $\left\Vert G\right\Vert _{p}.$ For example, the well-known
inequality $\lambda\geq2m/n$ yields a generalization of another bound of
Koolen and Moulton \cite{KoMo01}:

\begin{proposition}
\label{th5}If $2>p\geq1,$ $m\geq n/2,$ and $G$ is a graph of order $n$ and
size $m,$ then
\begin{equation}
\left\Vert G\right\Vert _{p}^{p}\leq\left(  2m/n\right)  ^{p}+\left(
n-1\right)  ^{1-p/2}\left(  2m-\left(  2m/n\right)  ^{2}\right)  ^{p/2}.
\label{bo2}%
\end{equation}
Equality holds if and only if $G$ is either $\left(  n/2\right)  K_{2},$ or
$K_{n},$ or a design graph.
\end{proposition}

However, inequality $\lambda\geq2m/n$ is just one of the infinitely many lower
bounds on $\lambda$ in terms of walks, see \cite{Nik06a} for details. Since
these lower bounds are stronger than $\lambda\geq2m/n,$ they imply infinitely
many upper bounds on $\left\Vert G\right\Vert _{p}.$\medskip

Next, we want an absolute bound on $\left\Vert G\right\Vert _{p}$ depending
just on the order of $G.$ We can use (\ref{bo2}) and some calculus to get such
a bound, but it will be very tangled. Instead, Theorem \ref{KMSmn} provides a
concise expression, which generalizes Koolen and Moulton's bound (\ref{KM}):

\begin{theorem}
\label{KMS}If $2>p\geq1$ and $G$ is a graph of order $n$ with adjacency matrix
$A,$ then
\begin{equation}
\left\Vert G\right\Vert _{p}\leq\frac{n^{1/p+1/2}}{2}+\frac{n}{2}. \label{bo5}%
\end{equation}
If $G$ is a graph of maximal energy, then
\[
\left\Vert G\right\Vert _{p}>\frac{n^{1/p+1/2}}{2}.
\]

\end{theorem}

Note that bound (\ref{bo5}) is never attained, but may be rather tight.

\subsubsection{Upper bounds for $p>2$}

Theorem \ref{MCabs1} gives the best possible bound on $\left\Vert A\right\Vert
_{p}$ of any matrix $A$ and any $p>2.$ This bound works also for graphs, but
unfortunately it is never exact, for adjacency matrices have zero diagonal.
Thus, the upper bounds on $\left\Vert G\right\Vert _{p}$ gives rise to some
subtle problems, which are not resolved yet.

We start with the following approximate result:

\begin{proposition}
\label{ps5}If $p>2$ and $G$ is a graph with $m$ edges, then
\begin{equation}
\left\Vert G\right\Vert _{p}^{p}<2m\left(  -\frac{1}{2}+\sqrt{2m+\frac{1}{4}%
}\right)  ^{p-2}. \label{SSp}%
\end{equation}
The inequality is tight, since for every $m,$ there is a graph with $m$ edges
satisfying
\begin{equation}
\left\Vert G\right\Vert _{p}>-\frac{3}{2}+\sqrt{2m+\frac{1}{4}}. \label{SSp1}%
\end{equation}

\end{proposition}

\begin{proof}
Let $p>2$ and $G$ be a graph of order $n$ and size $m$, that is to say
\begin{equation}
\sigma_{1}^{2}\left(  G\right)  +\cdots+\sigma_{n}^{2}\left(  G\right)  =2m.
\label{rest1}%
\end{equation}
We want to maximize $\left\Vert G\right\Vert _{p}^{p}$ subject to
(\ref{rest1}). First, note that
\begin{equation}
\sigma_{1}^{p}\left(  G\right)  +\cdots+\sigma_{n}^{p}\left(  G\right)
\leq\sigma_{1}^{p-2}\left(  G\right)  \left(  \sigma_{1}^{2}\left(  G\right)
+\cdots+\sigma_{n}^{2}\left(  G\right)  \right)  =\sigma_{1}^{p-2}\left(
G\right)  2m. \label{in7}%
\end{equation}
On the other hand, Stanley's bound \cite{Sta87} implies that
\begin{equation}
\sigma_{1}\left(  G\right)  =\lambda_{1}\left(  G\right)  \leq-\frac{1}%
{2}+\sqrt{2m+\frac{1}{4}}, \label{in8}%
\end{equation}
and (\ref{SSp}) follows. The inequality is strict, because if equality holds
in (\ref{in8}), then $G$ is a complete graph with possibly isolated vertices,
and in that case inequality (\ref{in7}) is strict.

To see the validity of (\ref{SSp1}), consider the maximal complete graph
$K_{s}$ that can be formed with at most $m$ edges. This means
\[
\binom{s+1}{2}>m
\]
and so%
\[
s>-\frac{1}{2}+\sqrt{2m+\frac{1}{4}}.
\]
Hence,
\[
\left\Vert K_{s}\right\Vert _{p}>s-1>-\frac{3}{2}+\sqrt{2m+\frac{1}{4}}.
\]
Clearly, inequality (\ref{SSp1}) shows that (\ref{SSp}) is asymptotically
tight, for the right side of (\ref{SSp}) satisfies
\[
2m\left(  -\frac{1}{2}+\sqrt{2m+\frac{1}{4}}\right)  ^{p-2}<\left(  2m\right)
^{p}.
\]
Proposition \ref{ps5} is proved.
\end{proof}

Clearly, Proposition \ref{ps5} is just a tight asymptotic result and equality
never holds in (\ref{SSp}). This observation prompts the following conjecture:

\begin{conjecture}
\label{con1}If $p>2$ and $G$ is a graph of size $m,$ then%
\[
\left\Vert G\right\Vert _{p}^{p}\leq\left(  -\frac{1}{2}+\sqrt{2m+\frac{1}{4}%
}\right)  ^{p}-\frac{1}{2}+\sqrt{2m+\frac{1}{4}}.
\]
Equality holds if and only if $G$ is a complete graph with possibly some
isolated vertices.
\end{conjecture}

Finally, we want an upper bound on $\left\Vert G\right\Vert _{p}$ in terms of
the order of $G.$ Thus, using (\ref{SSp}), we see that if $G$ is a graph of
order $n,$ then
\begin{equation}
\left\Vert G\right\Vert _{p}^{p}<\left(  n-1\right)  ^{p}+\left(  n-1\right)
^{p-1}. \label{uba}%
\end{equation}
Since $\left\Vert K\right\Vert _{p}^{p}=\left(  n-1\right)  ^{p}+n-1,$ bound
(\ref{uba}) is asymptotically tight. However, equality never holds in
(\ref{uba}), and the following conjecture is plausible:

\begin{conjecture}
\label{con2}If $p>2$ and $G$ is a graph of order $n$, then%
\[
\left\Vert G\right\Vert _{p}^{p}<\left\Vert K\right\Vert _{p}^{p},
\]
unless $G=K_{n}.$
\end{conjecture}

Note that if Conjecture \ref{con1} is true, then Conjecture \ref{con2} is true
as well.

\subsubsection{Lower bounds}

For $q=2$ Corollary \ref{p3g} gives the lower bound:

\begin{corollary}
\label{KMlo}Let $p>2.$ If $G$ is a graph of order $n,$ size $m,$ and largest
eigenvalue $\lambda,$ then
\begin{equation}
\left\Vert G\right\Vert _{p}^{p}\geq\lambda^{p}+\left(  n-1\right)
^{1-p/2}\left(  2m-\lambda^{2}\right)  ^{p/2}. \label{Nlo}%
\end{equation}
Equality holds in (\ref{Nlo}) if and only if $\sigma_{2}\left(  G\right)
=\cdots=\sigma_{n}\left(  G\right)  .$
\end{corollary}

As with all bounds of this type, Proposition \ref{KMlo} is most valuable as a
device to convert lower bounds on $\lambda\left(  G\right)  $ into lower
bounds on $\left\Vert G\right\Vert _{p}$; hence, following the argument after
Proposition \ref{th5}, we come up with infinitely many lower bounds on
$\left\Vert G\right\Vert _{p}$.

In particular, we get a twin proposition of Proposition \ref{th5}, yielding a
lower bound this time:

\begin{theorem}
\label{Nlb}Let $p>2$ and $m\geq n/2.$ If $G$ is a graph of order $n$ and size
$m,$ then
\begin{equation}
\left\Vert G\right\Vert _{p}^{p}\geq\left(  2m/n\right)  ^{p}+\left(
n-1\right)  ^{1-p/2}\left(  2m-\left(  2m/n\right)  ^{2}\right)  ^{p/2}.
\label{Nlo1}%
\end{equation}
Equality holds if and only if $G$ is either $\left(  n/2\right)  K_{2},$ or
$K_{n},$ or a design graph.
\end{theorem}

Theorem \ref{Nlb} has implications outside of the study of graph energy, e.g.,
the following extremal result holds:

\begin{corollary}
Let $k\geq2,$ $m\geq n/2,$ and let $G$ be a graph of order $n$ and size $m.$
Then the number of closed walks of length $2k$ of $G$ is at least
\[
\frac{1}{4k}\left(  \left(  2m/n\right)  ^{2k}+\left(  n-1\right)
^{1-k}\left(  2m-\left(  2m/n\right)  ^{2}\right)  ^{k}\right)  ,
\]
with equality holding if and only if $G$ is either $\left(  n/2\right)
K_{2},$ or $K_{n},$ or a design graph.
\end{corollary}

Finally, we note a natural extension of the simple bound (\ref{lobo}):

\begin{corollary}
If $2>p\geq1$ and $G$ is a nonempty graph with $m$ edges and largest
eigenvalue $\lambda$, then%
\[
\left\Vert G\right\Vert _{p}^{p}\geq\lambda^{p}+\frac{2m-\lambda^{2}}%
{\sigma_{2}^{2-p}\left(  G\right)  }.
\]
Equality holds if and only if all nonzero eigenvalues of $G$ other than
$\lambda$ have the same absolute value.
\end{corollary}

\subsection{\label{shin}Bounds based on H\"{o}lder's inequality}

In recent years much attention has been paid to the "spectral moments method"
in graph energy (see, e.g.,, \cite{PMR05},\cite{RaTi04},\cite{Zho04}%
,\cite{ZGPRM07},\cite{GLS12}). In this section we study this topic in the
light of H\"{o}lder's inequality and use Schatten norms to express the results.

Recall the classical H\"{o}lder inequality: \emph{Let }$x=\left[
x_{i}\right]  $\emph{ and }$y=\left[  y_{i}\right]  $\emph{ be real nonzero
vectors. If the positive numbers }$s$\emph{ and }$t$\emph{ satisfy
}$1/s+1/t=1,$ \emph{then }%
\begin{equation}
x_{1}y_{1}+\cdots+x_{n}y_{n}\leq\left\vert \left[  x_{i}\right]  \right\vert
_{s}\left\vert \left[  y_{i}\right]  \right\vert _{t}.\label{Holdin}%
\end{equation}
\emph{If equality holds, then }$\left(  \left\vert x_{1}\right\vert
^{s},,\ldots,\left\vert x_{n}\right\vert ^{s}\right)  $\emph{ and
}$(\left\vert y_{1}\right\vert ^{t},\ldots,\left\vert y_{n}\right\vert ^{t}%
)$\emph{ are collinear.}

H\"{o}lder's inequality easily implies the following simple result:

\begin{proposition}
\label{proP}Let $p>0,$ $q>0,$ $\alpha>0,$ $\beta>0,$ and $\alpha+\beta=1.$ If
$a_{1},\ldots,a_{n}$ are positive numbers, then%
\[
a_{1}^{\alpha p+\beta q}+\cdots+a_{n}^{\alpha p+\beta q}\leq\left(  a_{1}%
^{p}+\cdots+a_{n}^{p}\right)  ^{\alpha}\left(  a_{1}^{q}+\cdots+a_{n}%
^{q}\right)  ^{\beta}%
\]
If $\alpha p\neq\beta q,$ then equality holds if and only if $a_{1}%
=\cdots=a_{n}$.
\end{proposition}

\begin{proof}
Let $\left(  x_{1},\ldots,x_{n}\right)  :=\left(  a_{1}^{\alpha p}%
,\ldots,a_{n}^{\alpha p}\right)  $ and $\left(  y_{1},\ldots,y_{n}\right)
:=(a_{1}^{\beta q},\ldots,a_{n}^{\beta q}).$ Applying (\ref{Holdin}) with
$s:=1/\alpha$ and \ $t:=1/\beta,$ we find that%
\[
a_{1}^{\alpha p+\beta q}+\cdots+a_{n}^{\alpha p+\beta q}\leq\left(  a_{1}%
^{p}+\cdots+a_{n}^{p}\right)  ^{\alpha}\left(  a_{1}^{q}+\cdots+a_{n}%
^{q}\right)  ^{\beta}.
\]
If equality holds, then $\left(  a_{1}^{\alpha p},\ldots,a_{n}^{\alpha
p}\right)  $ is collinear to $(a_{1}^{\beta q},\ldots,a_{n}^{\beta q}).$ If
$\alpha p\neq\beta q,$ this can happen if and only if $a_{1}=\cdots=a_{n}.$
\end{proof}

Now we gat a fairly abstract inequality about the Schatten norms of arbitrary matrices:

\begin{corollary}
\label{proPP}Let $p>0,$ $q>0,$ $\alpha>0,$ $\beta>0,$ and $\alpha+\beta=1.$ If
$A$ is a matrix, then
\[
\left\Vert A\right\Vert _{p}^{\alpha p}\left\Vert A\right\Vert _{q}^{\beta
q}\geq\left\Vert A\right\Vert _{\alpha p+\beta q}^{\alpha p+\beta q}%
\]
If $\alpha p\neq\beta q,$ then equality holds if and only if all nonzero
singular values of $A$ are equal.
\end{corollary}

Slightly weaker versions of Proposition \ref{proP} were proved in \cite{PMR05}
and \cite{ZGPRM07}. The case $p=4,$ $q=1,$ $\alpha=1/3,$ and $\beta=2/3$ of
Proposition \ref{proP} was proved by Rada and Tineo in \cite{RaTi04}, and by
Zhou in \cite{Zho04}, and later rediscovered several times:

\begin{corollary}
If $a_{1},\ldots,a_{n}$ are nonnegative numbers, then%
\[
(a_{1}^{2}+\cdots+a_{n}^{2})^{3}\leq(a_{1}+\cdots+a_{n})^{2}(a_{1}^{4}%
+\cdots+a_{n}^{4}).
\]
Equality holds if and only if all nonzero number among $a_{1},\ldots,a_{n}$
are equal.
\end{corollary}

Rada and Tineo \cite{RaTi04} and Zhou \cite{Zho04} also deduced the following corollary:

\begin{corollary}
Let $G$ be a graph with $m$ edges and $q$ closed walks of length four. Then%
\[
q\left\Vert G\right\Vert _{\ast}^{2}\geq m^{3}.
\]
Equality holds if and only if $G$ is a disjoint union of isolated vertices and
complete bipartite graphs with the same number of edges.
\end{corollary}

Clearly, Corollary \ref{proPP} leads to the following generalization of the
above result:

\begin{corollary}
\label{cor4}Let $p>0,$ $q>0,$ $\alpha>0,$ $\beta>0,$ and $\alpha+\beta=1.$ If
$A$ is a matrix, then
\[
\left\Vert G\right\Vert _{p}^{\alpha p}\left\Vert G\right\Vert _{q}^{\beta
q}\geq\left\Vert G\right\Vert _{\alpha p+\beta q}^{\alpha p+\beta q}.
\]
Equality holds if and only if $G$ is a disjoint union of isolated vertices and
complete bipartite graphs with the same number of edges.
\end{corollary}

\subsection{\label{SSr}The Schatten norms of $r$-partite matrices and graphs}

In this section we study the Schatten norms of $r$-partite graphs and
matrices, extending some results about the trace norm obtained in
\cite{Nik15a} and presented in Section \ref{secMTN}. These new extensions are
published here for the first time, so we shall give full proofs.\medskip

Write $T_{r}\left(  n\right)  $ for the complete $r$-partite graph of order
$n$ with vertex classes of size $\left\lfloor n/r\right\rfloor $ or
$\left\lceil n/r\right\rceil ,$ and recall that $T_{r}\left(  n\right)  $ is
known as the $r$\emph{-partite Tur\'{a}n graph of order }$n.$ The number of
edges $t_{r}\left(  n\right)  $ of $T_{r}\left(  n\right)  $ is well studied
in connection with the Tur\'{a}n theorem \cite{Tur41}; in particular, it is
known that%
\[
\frac{r-1}{2r}n^{2}-\frac{r}{8}\leq t_{r}\left(  n\right)  \leq\frac{r-1}%
{2r}n^{2}.
\]

\subsubsection{$r$-partite matrices}

We start with a bound on the pivotal Schatten $2$-norm $\left\Vert
A\right\Vert _{2}$ (the Frobenius norm) of an $r$-partite matrix $A$:

\begin{proposition}
\label{pro2}Let $n\geq r\geq2,$ and let $A=\left[  a_{i,j}\right]  $ be an
$n\times n$ matrix with $\left\Vert A\right\Vert _{\max}\leq1.$ If $A$ is
$r$-partite, then%
\begin{equation}
\left\Vert A\right\Vert _{2}\leq\sqrt{2t_{r}\left(  n\right)  }. \label{te}%
\end{equation}
Equality holds if and only if the matrix $\left\vert A\right\vert =\left[
\left\vert a_{i,j}\right\vert \right]  $ is the adjacency matrix of
$T_{r}\left(  n\right)  $.
\end{proposition}

\begin{proof}
Let $A$ be an $r$-partite matrix satisfying the premises and let $\left[
n\right]  =N_{1}\cup\cdots\cup N_{r}$ be a partition such that $A\left[
N_{i},N_{i}\right]  =0$ for any $i\in\left[  r\right]  $.

Define a graph $G\ $with $V\left(  G\right)  =\left[  n\right]  $ such that
$i$ is adjacent to $j$ whenever $\left\vert a_{i,j}\right\vert +\left\vert
a_{j,i}\right\vert \neq0.$ Since the sets $N_{1},\ldots,N_{r}$ induce no edges
in $G$, we see that $G$ is $r$-partite. Now Tur\'{a}n's theorem \cite{Tur41}
implies that
\[
2e\left(  G\right)  <2t_{r}\left(  n\right)  ,
\]
unless $G=T_{r}\left(  n\right)  .$ Hence,
\[
\left\Vert A\right\Vert _{2}^{2}=\sum_{i,j\in\left[  n\right]  }\left\vert
a_{i,j}\right\vert ^{2}\leq\sum_{i,j\in\left[  n\right]  }\left\vert
a_{i,j}\right\vert \leq\sum_{\left\{  i,j\right\}  \in E\left(  G\right)
}\left\vert a_{i,j}\right\vert +\left\vert a_{j,i}\right\vert \leq2e\left(
G\right)  \leq2t_{r}\left(  n\right)  .
\]
If equality holds, then $G=$ $T_{r}\left(  n\right)  ,$ and $\left\vert
a_{i,j}\right\vert =\left\vert a_{j,i}\right\vert =1$ for every edge $\left\{
i,j\right\}  \in E\left(  G\right)  $. Hence $\left\vert A\right\vert $ is the
adjacency matrix of $T_{r}\left(  n\right)  $.

If $\left\vert A\right\vert $ is the adjacency matrix of $T_{r}\left(
n\right)  ,$ then $\left\Vert A\right\Vert _{2}^{2}=\left\Vert \left\vert
A\right\vert \right\Vert _{2}^{2}=2t_{r}\left(  n\right)  ,$ and so $A$ forces
equality in (\ref{te}).
\end{proof}

Now, we easily get the following corollary:

\begin{corollary}
\label{cor}Let $n\geq r\geq2,$ and let $A=\left[  a_{i,j}\right]  $ be an
$n\times n$ matrix with $\left\Vert A\right\Vert _{\max}\leq1.$ If $A$ is
$r$-partite, then
\[
\left\Vert A\right\Vert _{2}\leq n\sqrt{1-1/r}.
\]
If $\left[  n\right]  =N_{1}\cup\cdots\cup N_{r}$ is a partition such that
$A\left[  N_{i},N_{i}\right]  =0$ for any $i\in\left[  r\right]  ,$ then
equality holds if and only if $\left\vert N_{1}\right\vert =\cdots=\left\vert
N_{r}\right\vert ,$ and $\left\vert a_{i,j}\right\vert =1$ whenever $i$ and
$j$ do not belong to the same partition set.
\end{corollary}

Armed with Corollary \ref{cor}, we are ready to give an upper bound on the
Schatten $p$-norm of a complex $r$-partite matrix whenever $2>p\geq1$.

\begin{theorem}
\label{ScMxr}Let $n\geq r\geq2,$ $2>p\geq1,$ and let $A=\left[  a_{i,j}%
\right]  $ be an $n\times n$ matrix with $\left\Vert A\right\Vert _{\max}%
\leq1.$ If $A$ is $r$-partite, then
\[
\left\Vert A\right\Vert _{p}\leq n^{1/2+1/p}\sqrt{1-1/r}.
\]
Equality holds if and only if all singular values of $A$ are equal to
$\sqrt{\left(  1-1/r\right)  n}$.
\end{theorem}

\begin{proof}
Let $A$ be an $r$-partite matrix satisfying the premises, and suppose that
$\left[  n\right]  =N_{1}\cup\cdots\cup N_{r}$ is a partition such that
$A\left[  N_{i},N_{i}\right]  =0$ for any $i\in\left[  r\right]  .$ Let
$\sigma_{1},\ldots,\sigma_{n}$ be the singular values of $A.$ Our starting
point is Proposition \ref{pro2}, which gives%
\[
\sigma_{1}^{2}+\cdots+\sigma_{n}^{2}=\sum_{i,j\in\left[  n\right]  }\left\vert
a_{i,j}\right\vert ^{2}\leq\left(  1-\frac{1}{r}\right)  n^{2}.
\]
This bound and the PM inequality imply that
\begin{equation}
\left(  \frac{\sigma_{1}^{p}+\cdots+\sigma_{n}^{p}}{n}\right)  ^{1/p}%
\leq\left(  \frac{\sigma_{1}^{2}+\cdots+\sigma_{n}^{2}}{n}\right)  ^{1/2}%
\leq\left(  1-1/r\right)  ^{1/2}n^{1/2}. \label{in6}%
\end{equation}
Hence,
\[
\left\Vert A\right\Vert _{p}=\left(  \frac{\sigma_{1}^{p}+\cdots+\sigma
_{n}^{p}}{n}\right)  ^{1/p}\leq n^{1/2+1/p}\sqrt{1-1/r},
\]
as stated.

If $\left\Vert A\right\Vert _{p}=n^{1/2+1/p}\sqrt{1-1/r},$ then the condition
for equality in the PM inequality used in (\ref{in6}) implies that
\[
\sigma_{1}=\cdots=\sigma_{n}=\sqrt{\left(  1-1/r\right)  n},
\]
completing the proof.
\end{proof}

Next, we show that if $r$ is the order of a conference matrix, then bound
(\ref{Mb}) in Theorem \ref{ScMxr} is best possible for infinitely many $n$.

\begin{theorem}
\label{Scth3}Let let $r$ be the order of a conference matrix of order $r,$ and
let $k$ be the order of an Hadamard matrix. If $p\geq1,$ there exists an
$r$-partite matrix $A$ of order $n=rk$ with $\left\Vert A\right\Vert _{\max
}=1$ such that
\[
\left\Vert A\right\Vert _{p}=n^{1/2+1/p}\sqrt{1-1/r}.
\]

\end{theorem}

\begin{proof}
Let $C\ $be a conference matrix of order $r$ and $H$ be an Hadamard matrix of
order $k.$ First, note that
\[
\left\Vert C\right\Vert _{p}=\left(  r\left(  \sqrt{r-1}\right)  ^{p}\right)
^{1/p}=r^{1/p}\sqrt{r-1},
\]
and%
\[
\left\Vert H\right\Vert _{p}=\left(  k\left(  \sqrt{k}\right)  ^{p}\right)
^{1/p}=k^{1/p+1/2}.
\]

Next, let $A:=C\otimes H,$ and partition $\left[  rk\right]  $ into $r$
consecutive segments $N_{1},\ldots,N_{r}$ of length $k;$ we see that
$\left\Vert A\right\Vert _{\max}=1$ and that $A\left[  N_{i},N_{i}\right]  =0$
for any $i\in\left[  r\right]  .$

Finally, we find that
\begin{align*}
\left\Vert A\right\Vert _{p}  &  =\left\Vert C\otimes H\right\Vert
_{p}=\left\Vert C\right\Vert _{p}\left\Vert H\right\Vert _{p}\\
&  =k^{1/p+1/2}r^{1/p}\sqrt{r-1}=\left(  kr\right)  ^{1/2+1/p}\sqrt{1-1/r}\\
&  =n^{1/2+1/p}\sqrt{1-1/r},
\end{align*}
completing the proof of Theorem \ref{Scth3}.
\end{proof}

We have no results about the Schatten $p$-norms of $r$-partite matrices for
$p>2$. The reader may be interested in the following conjecture:

\begin{conjecture}
Let $p>2,$ $n\geq r\geq2,$ and let $A=\left[  a_{i,j}\right]  $ be an $n\times
n$ complex matrix with $\left\Vert A\right\Vert _{\max}\leq1.$ If $A$ is
$r$-partite, then%
\[
\left\Vert A\right\Vert _{p}\leq\left\Vert T_{r}\left(  n\right)  \right\Vert
_{p}.
\]
Equality holds if and only if the matrix $\left\vert A\right\vert =\left[
\left\vert a_{i,j}\right\vert \right]  $ is the adjacency matrix of
$T_{r}\left(  n\right)  $.
\end{conjecture}

\subsubsection{$r$-partite graphs}

Similarly to Problem \ref{pr}, we raise the following natural problem about graphs:

\begin{problem}
\label{Scpr}What is the maximum Schatten $p$-norm of an $r$-partite graph of
order $n$?
\end{problem}

Our starting point is Theorem \ref{ScMxr}, which implies a bound for
nonnegative matrices, in particular, for graphs.

To simplify the subsequent proofs, first we give a uniform bound on the
Schatten $p$-norm of a complete multipartite graph.

\begin{proposition}
\label{pro3}If $p\geq1$ and $K$ is a complete $r$-partite graph of order $n$,
then%
\begin{equation}
\left\Vert K\right\Vert _{p}\leq2\left(  1-1/r\right)  n. \label{bo4}%
\end{equation}

\end{proposition}

\begin{proof}
Indeed, Propositions \ref{ps2} and \ref{ps3} imply that $\left\Vert
K\right\Vert _{p}\leq\left\Vert K\right\Vert _{\ast},$ and the result of
Caporossi, Cvetkovi\'{c}, Gutman, and Hansen \cite{CCGH99} implies that
$\left\Vert K\right\Vert _{\ast}=2\lambda_{1}\left(  K\right)  .$ Finally,
recall that Cvetkovi\'{c} \cite{Cve72} showed that $\lambda_{1}\left(
K\right)  \leq\left(  1-1/r\right)  n,$ completing the proof.
\end{proof}

Note that bound (\ref{bo4}) is best possible for $p=1,$ and is within a factor
of $2$ of the best possible for any $p>1.$ Indeed, if $K$ is a complete
regular $r$-partite graph, then
\[
\left\Vert K\right\Vert _{p}>\lambda_{1}\left(  K\right)  =\left(
1-1/r\right)  n.
\]

\begin{theorem}
\label{Sth2}Let $n\geq r\geq2$, $2>p\geq1,$ and let $A$ be an $n\times n$
nonnegative matrix with $\left\Vert A\right\Vert _{\max}\leq1.$ If $A$ is
$r$-partite, then
\[
\left\Vert A\right\Vert _{p}\leq\frac{1}{2}n^{1/2+1/p}\sqrt{1-1/r}+\left(
1-1/r\right)  n.
\]

\end{theorem}

\begin{proof}
Let $A$ be an $r$-partite matrix satisfying the premises, and suppose that
$\left[  n\right]  =N_{1}\cup\cdots\cup N_{r}$ is a partition such that
$A\left[  N_{i},N_{i}\right]  =0$ for any $i\in\left[  r\right]  .$

For each $i\in\left[  r\right]  ,$ set $n_{i}:=\left\vert N_{i}\right\vert ,$
and write $K$ for the matrix obtained from $J_{n}$ by zeroing the submatrices
$J_{n}\left[  N_{i},N_{i}\right]  $ for each $i\in\left[  r\right]  .$ Note
that $K$ is the adjacency matrix of the complete $r$-partite graph with vertex
classes $N_{1},\ldots,N_{r}.$

Now, let $B:=2A-K,$ and note that the matrix $B$ and the sets $N_{1}%
,\ldots,N_{r}$ satisfy the premises of Theorem \ref{ScMxr}; hence the triangle
inequality implies that%
\begin{align*}
n^{1/2+1/p}\sqrt{1-1/r}  &  \geq\left\Vert B\right\Vert _{p}\geq\left\Vert
2A-K\right\Vert _{p}\\
&  \geq2\left\Vert A\right\Vert _{p}-\left\Vert K\right\Vert _{p}%
\geq2\left\Vert A\right\Vert _{p}-2\left(  1-1/r\right)  n,
\end{align*}
completing the proof of Theorem \ref{th2}.
\end{proof}

Note that the matrix $A$ in Theorems \ref{ScMxr} and \ref{Sth2} needs not be
symmetric; nonetheless, the following immediate corollary is tight up to an
additive term that is linear in $n$:

\begin{corollary}
\label{Scor1}Let $n\geq r\geq2$ and $2>p\geq1.$ If $G\ $is an $r$-partite
graph of order $n,$ then
\begin{equation}
\left\Vert G\right\Vert _{p}\leq\frac{1}{2}n^{1/2+1/p}\sqrt{1-1/r}+\left(
1-1/r\right)  n. \label{Seab}%
\end{equation}

\end{corollary}

To prove the tightness of Theorems \ref{Sth2} and Corollary \ref{Scor1}, we
modify the construction in Theorem \ref{Scth3} as follows:

\begin{theorem}
\label{Sth4}Let $2>p\geq1$ and let $r$ be the order of a real symmetric
conference matrix. If $k$ is the order of a real symmetric Hadamard matrix,
then there is an $r$-partite graph $G$ of order $n=rk$ such that
\[
\left\Vert G\right\Vert _{p}\geq\frac{1}{2}n^{1/2+1/p}\sqrt{1-1/r}-\left(
1-1/r\right)  n.
\]

\end{theorem}

\begin{proof}
Let $C$ be a real symmetric conference matrix of order $r$, and let $H$ be a
real symmetric Hadamard matrix of order $k$. Let $B:=C\otimes H,$ and
partition $\left[  rk\right]  $ into $r$ consecutive segments $N_{1}%
,\ldots,N_{r}$ of length $k.$

We see that $B\left[  N_{i},N_{i}\right]  =0$ for any $i\in\left[  r\right]
,$ and that $B\left[  N_{i},N_{j}\right]  $ is a $\left(  -1,1\right)
$-matrix for any $i,j\in\left[  r\right]  $ with $i\neq j.$ Finally, set
$K_{t}:=J_{t}-I_{t}$ and let
\[
A:=\frac{1}{2}\left(  B+K_{r}\otimes J_{k}\right)  .
\]

Note that $A$ is a symmetric $\left(  0,1\right)  $-matrix, and $A\left[
N_{i},N_{i}\right]  =0$ for any $i\in\left[  r\right]  $. Hence $A$ is the
adjacency matrix of an $r$-partite graph $G\ $of order $n.$

Next, note that the singular values of $B$ are equal to
\[
\sqrt{k\left(  r-1\right)  }=\sqrt{\left(  1-1/r\right)  n}.
\]
Thus, the triangle inequality implies that
\[
\left\Vert \left(  B+K_{r}\otimes J_{k}\right)  \right\Vert _{p}\geq\left\Vert
B\right\Vert _{p}-\left\Vert K_{r}\otimes J_{k}\right\Vert _{p}\geq
n^{1/2+1/p}\sqrt{1-1/r}-2\left(  1-1/r\right)  n,
\]
and so,
\[
\left\Vert G\right\Vert _{\ast}\geq\frac{1}{2}n^{1/2+1/p}\sqrt{1-1/r}-\left(
1-1/r\right)  n,
\]
completing the proof of Theorem \ref{Sth4}.
\end{proof}

Concrete examples of the above construction can be found using, e.g., Paley's
conference and Hadamard matrices described in Section \ref{secMTN}.

If $p=2,$ let us note the following simple result of extremal graph
theory:\medskip

\emph{If }$n\geq r$\emph{ and }$G$\emph{ is an }$r$\emph{-partite graph of
order }$n,$\emph{ then} $\left\Vert G\right\Vert _{2}<\left\Vert T_{r}\left(
n\right)  \right\Vert _{2},$ \emph{unless }$G=T_{r}\left(  n\right)  $%
\emph{.}\medskip

However, it is not so easy to generalize this statement for $p>2.$ The
following conjecture seems plausible:

\begin{conjecture}
If $p>2,$ $n\geq r,$ and $G$ is an $r$-partite graph of order $n,$ then%
\[
\left\Vert G\right\Vert _{p}<\left\Vert T_{r}\left(  n\right)  \right\Vert
_{p},
\]
unless $G=T_{r}\left(  n\right)  .$
\end{conjecture}

\subsection{\label{Snt}The Schatten norms of trees}

In \cite{Csi10}, Csikv\'{a}ri showed that the path has the minimum number of
closed walks of any given length among all connected graphs of given order,
and the star has the maximum number of closed walks of any given length among
all trees of given order.

These results can be expressed in Schatten norms as follows:

\begin{proposition}
If $G$ is a connected graph of order $n,$ then $\left\Vert G\right\Vert
_{2k}\geq\left\Vert P_{n}\right\Vert _{2k}$ for every integer $k\geq2.$
\end{proposition}

\begin{proposition}
If $T$ is a tree of order $n,$ then $\left\Vert T\right\Vert _{2k}%
\leq\left\Vert S_{n}\right\Vert _{2k}$ for every integer $k\geq2.$
\end{proposition}

On the other hand, it is known that the path has maximal energy among all
trees of given order and the star has minimal energy among all connected
graphs of given order. These facts lead us to the following natural questions:

\begin{question}
Let $G$ be a connected graph of order $n.$

(a) Is it true that $\left\Vert G\right\Vert _{p}\geq\left\Vert P_{n}%
\right\Vert _{p}$ for every $p>2?$

(b) Is it true that $\left\Vert G\right\Vert _{p}\geq\left\Vert S_{n}%
\right\Vert _{p}$ for every $1<p<2?$
\end{question}

\begin{question}
Let $T$ be a tree of order $n.$

(a) Is it true that%
\[
\left\Vert S_{n}\right\Vert _{p}\geq\left\Vert T\right\Vert _{p}\geq\left\Vert
P_{n}\right\Vert _{p}%
\]
for every $p>2?$

(b) Is it true that
\[
\left\Vert P_{n}\right\Vert _{p}\geq\left\Vert T\right\Vert _{p}\geq\left\Vert
S_{n}\right\Vert _{p}%
\]
for every $1<p<2?$
\end{question}

Let us note that the techniques that work for the norm $\left\Vert
G\right\Vert _{2k}$ for integer $k$ are not directly applicable to the above questions.

\subsection{\label{sec1.1}The Schatten norms of almost all graphs}

Our last topic provides some insight into the Schatten $p$-norm of the
\textquotedblleft average\textquotedblright\ graph of order $n$. In
\cite{Nik12}, the following theorem was proved:

\begin{theorem}
\label{Rands}If $G$ is a graph of order $n$, then with probability tending to
$1:$

(i) if $2>p\geq1,$ then
\begin{equation}
\left\Vert G\right\Vert _{p}=\left(  \frac{1}{\sqrt{\pi}}\cdot\frac
{\Gamma\left(  p/2+1/2\right)  }{\Gamma\left(  p/2+2\right)  }+o\left(
1\right)  \right)  ^{1/p}n^{1/p+1/2}\text{;} \label{wig}%
\end{equation}

(ii) if $p=2,$ then%
\[
\left\Vert G\right\Vert _{p}=\left(  1/\sqrt{2}+o\left(  1\right)  \right)
n\text{;}%
\]

(iii) if $p>2,$ then%
\[
\left\Vert G\right\Vert _{p}=\left(  1/2+o\left(  1\right)  \right)  n.
\]

\end{theorem}

The proof is based on the assumption that the spectrum of the
\textquotedblleft average graph\textquotedblright\ of order $n$ is
approximated by the spectrum of the Erd\H{o}s-R\'{e}nyi random graph $G\left(
n,1/2\right)  .$ Recall that in $G\left(  n,1/2\right)  $ every pair of
vertices is joined independently with probability $1/2.\medskip$

If $p>2,$ the value of $\left\Vert G\right\Vert _{p}$ is determined
essentially by the largest eigenvalue of $G,$ which is almost surely
$n/2+o\left(  n\right)  .$ This implies (iii).

If $p=2,$ then $\left\Vert G\right\Vert _{2}^{2}$ is equal to twice the number
of edges of $G$, which is almost surely $\left(  1/2+o\left(  1\right)
\right)  n^{2},$ implying \ (ii)$.$

Finally, for $2>p\geq1,$ one can use Wigner's semicircle law and calculate the
required value, getting (i).\medskip

The three cases $2>p\geq1,$ $p=2$, and $p>2$ are quite disparate and seem to
contradict Proposition \ref{ps1}, which claims that $\left\Vert G\right\Vert
_{p}$ is differentiable in $p$ for $p\geq1.$ However, in Theorem \ref{Rands}
it is supposed that $p$ is fixed and $n\rightarrow\infty.$ Thus, the three
instances of the term $o\left(  1\right)  $ depend on $p$ and are different in
each case.\medskip

Let us note that calculating the right side of (\ref{wig}) for $p=1,$ we see
that the energy of almost all graphs of order $n$ is almost surely equal to
\[
\left(  \frac{4}{3\pi}+o\left(  1\right)  \right)  n^{3/2},
\]
This basic fact has been established in \cite{Nik07i}.\bigskip

\textbf{Concluding remark\medskip}

We hope to have shown that matrix norms of graphs are a vital and challenging
topic, which throws bridges between analysis and combinatorics. Also, we hope
to have shown that the core research on graph energy has strong ties to
mainstream mathematics. Let us hope that by exploring and expanding these
connections, future research on graph energy shall keep flourishing.

\end{document}